\documentclass[12pt,a4paper, twoside]{article}

\usepackage{amsmath,amsthm}
\usepackage{amssymb}
\usepackage{amsfonts,amsmath,amsthm, amssymb, mathrsfs}
\usepackage{enumerate}
\usepackage{indentfirst}
\usepackage[T1]{fontenc}
\newtheorem{theorem}{Theorem}[section]
\newtheorem*{theorem2}{Theorem A}
\newtheorem{corollary}[theorem]{Corollary}
\newtheorem{lemma}[theorem]{Lemma}
\newtheorem{proposition}[theorem]{Proposition}

\newtheorem{theore}{Theorem}[section]

\theoremstyle{definition}
\newtheorem{definition}[theorem]{Definition}
\newtheorem{remark}[theorem]{Remark}
\newtheorem{example}[theorem]{Example}
\numberwithin{equation}{section}
\usepackage[none]{hyphenat}

\usepackage[usenames]{color}
\usepackage[colorlinks,linkcolor=red,anchorcolor=blue,citecolor=blue]{hyperref}
\begin{document}

	\baselineskip=17pt
	
	%%%%%%%%%%%%%%%%

	\title{Bowen's  equations for upper metric mean dimension with potential}
	
	%%\author{\\
	%%	School of Mathematical Sciences\\
	   %%Nanjing Normal University\\ 
	  %% Jiangsu, 210046, PR China\\ 
		%}
		\author{Rui Yang$^{1}$, Ercai Chen$^{1}$ and Xiaoyao Zhou*$^{1}$\\
			\small 1 School of Mathematical Sciences and Institute of Mathematics, Nanjing Normal University,\\
			\small Nanjing 210023, Jiangsu, P.R.China\\
			\small    e-mail: zkyangrui2015@163.com; ecchen@njnu.edu.cn; zhouxiaoyaodeyouxian@126.com\\ }
	
	\date{}
	
	\maketitle
	
	%% Classification and key words; note that the 2010 classification is used:
	
	\renewcommand{\thefootnote}{}
	
	
	
	\footnote{*corresponding author}
	
	\renewcommand{\thefootnote}{\arabic{footnote}}
	\setcounter{footnote}{0}
	
	%%%%%%%%
	
	\begin{abstract}
   Firstly, we  introduce a new notion called induced  upper metric  mean dimension  with potential, which  naturally  generalizes the definition of  upper metric mean dimension with potential  given by Tsukamoto to more general cases, then we  establish  variational principles  for it in terms of  upper  and lower  rate distortion dimensions and show   there exists  a Bowen's  equation   between induced  upper metric  mean dimension  with potential and  upper metric  mean dimension with potential.
  
    Secondly, we continue to  introduce  two new notions, called  BS  metric mean dimension and Packing BS metric mean dimension  on  arbitrary subsets, to  establish  Bowen's equations for Bowen upper  metric mean dimension  and Packing upper metric mean dimension with potential on subsets.  Besides,  we also obtain two  variational principles for  BS  metric mean dimension and Packing BS metric mean dimension on subsets. 
    
    Finally, the special interest about the Bowen  upper   metric  mean dimension of  the set of generic points of  ergodic measures are also involved. 
	\end{abstract}
	%%%%%%%%%%%%%%%%%%%%%%%%%%%%%%%%%%%%%%%%%%%%%%%%%%%%%
	\section{Introduction}

Mean topological dimension introduced  by Gromov \cite{gromov}  is a new  topological  invariant  in  topological dynamical systems. Later,  Lindenstrauss and Weiss \cite{lw00}  introduced the notion called metric mean dimension   to capture the complexity of infinite topological entropy systems and revealed the well-known fact that  metric mean dimension  is an upper bound of mean topological dimension. Therefore,  metric mean dimension plays a vital role in dimension theory and deserves some special attentions. Very recently, Lindenstrauss  and Tsukamoto's   pioneering work \cite{lt18} showed  a  first important  relationship  between  mean dimension  theory  and  ergodic  theory, which is an analogue of classical variational  principle  for topological entropy. More discussions associated with this result can be  found in  \cite{gs20,chen}.   From that time on,  Lindenstrauss  and Tsukamoto's work inspired more and more researchers   to  inject  ergodic theoretic ideas  into mean dimension theory by constructing  some new variational principles, and we refer to \cite{vv17, lt19, t20, gs20, shi, cls21, w21} for more details.
Before stating our main results, we list some basic notions and recall some necessary backgrounds. 
	
By a pair $(X,f)$ we mean a  topological dynamical system (TDS for short), where  $X$ is a compact  metrizable  topological space and $ f$ is a  continuous self-map on $X$.   The set of metrics on $X$  compatible with the topology is denoted by $\mathscr D(X)$. We denote by  $C(X,\mathbb{R})$  the set of all  real-valued continuous functions of $X$ equipped  with the  supremum norm.  By $M(X),M(X,f),E(X,f)$  we denote  the sets of  all Borel probability measures on $X$, all $f$-invariant Borel probability measures on $X$, all ergodic measures on $X$, respectively.

In the setting of quasi-circles, Bowen \cite{b79} firstly  found  the Hausdorff dimension of  certain compact set is exactly  the unique root of the equation defined by the  topological pressure  of geometric potential function, which   was later known as Bowen's equation. In 2000,  Barreira and Schmeling  \cite{bs00} introduced  the notion of BS dimension (or  called u-dimension in that  paper) on subsets and proved that  BS dimension  is  the unique root of the equation defined by topological pressure  of additive potential function. The  non-additive setting   and  non-uniform setting  about  Bowen's equation can be found in \cite{b96} and \cite{c11}, respectively. Later, Xing  and Chen \cite{xc}  extended  the work of  Jaerisch et al. \cite{jms14}  to  general topological dynamical systems and introduced a notion called  induced topological pressure that specializes the  BS dimension,  and  they revealed   an important link between the induced topological pressure and the classical topological pressure is Bowen's  equation. Based on these work, the first purpose of  this paper is to  establish Bowen's equation  for upper metric mean dimension with potential on the whole phase space.

One says that a   topological dynamical system $(X, f)$ admits \emph{marker property}  if for any $N>0$, there exists an open set $U\subset X$  with property that  
$$U\cap f^n U=\emptyset, 1\leq n\leq N,  \mbox{and}~X=\cup_{n\in \mathbb{Z}} f^n U.$$

The  symbolic  version of marker property  was first introduced  by Krieger  in  \cite{k82} and then   the simplified version  of  Krieger's  marker lemma was given in \cite[Lemma 2.2]{b83}, and the  non-symbolic version of  marker property was  defined  in \cite[Definition 2]{d06} based on the  Krieger's marker lemma \cite{k82,b83}. 
For example,  free (no  aperiodic points) minimal systems  and their extensions \cite[Lemma 3.3]{l99}, an aperiodic finite-dimensional  TDS \cite[Theorem 6.1]{g15} and an extension of an aperiodic TDS  which has a countable number of minimal subsystems \cite[Theorem 3.5]{g17}  have marker property. In general,  the marker property  implies the aperiodicity  and  whether the aperiodicity implies the marker property or not is still an open problem   posed by Gutman in  \cite[Problem 5.4]{g15} and \cite[Problem 3.4]{g17}. As an application,  marker property  have been extensively  used  to  deal with the  embedding problems, readers can  turn to  \cite{l99,g15,glt16,g17,lt19,t20,gt20} for more details of this aspect.
	 
Tsukamoto \cite{t20} introduced a notion called upper metric mean dimension with  potential and proved the following 

\begin{theorem2}\label{thm A}
	 Let (X, f)  be a TDS admitting the marker property.  Then for all $d \in \mathscr{D}^{'}(X)$,
	 \begin{align*}
	 \overline{{mdim}}_M(X,f,d,\varphi)&=\sup_{\mu\in M(X,f)}\left(\underline{rdim}(X,f,d,\mu)+\int_X \varphi d\mu\right)\\
	 &=\sup_{\mu\in M(X,f)}\left(\overline{rdim}(X,f,d,\mu)+\int_X \varphi d\mu\right),
	 \end{align*}
where
 $\mathscr{D}^{'}(X)=\{d\in \mathscr{D}(X): mdim(X,f,\varphi)=\overline{{mdim}}_M(X,f,d,\varphi)\}$, $mdim(X,f,\varphi)$\\ denotes mean dimension with potential $\varphi$, see \rm{\cite[Subsection 1.2]{t20} }for its explicit definition. 
$\overline{{mdim}}_M(X,f,d,\varphi)$ is  upper metric mean dimension with potential $\varphi$ given in  subsection 2.1. $\underline{rdim}(X,f,d,\mu) $ and $ \overline{rdim}(X,f,d,\mu)$  respectively denote  lower and upper rate distortion dimensions, see \cite[Section 2]{t20} for  their  precise definitions.
\end{theorem2}

We remark that Theorem A can be  directly  deduced from \cite[Corollary 1.7, Theorem 1.8]{t20}. Here,  we borrow some ideas from \cite{jms14,xc} to define induced upper metric mean dimension with potential and establish a  Bowen's equation for  upper  metric mean dimension with potential on the whole phase space.

\begin{theorem}\label{thm 1.1}
	  
Let $(X,f)$ be a  TDS   with a  metric $d\in \mathscr{D}(X)$ and $\varphi, \psi \in C(X,\mathbb{R})$  with $\psi >0$.  Suppose that $\overline{mdim}_{M}(X,f,d,\varphi)<\infty$.
Then  $\overline{mdim}_{M,\psi}(X,f,d,\varphi)$ is the unique root of the equation $\overline{mdim}_{M}(X,f,d,\varphi-\beta\psi)=0,$
where $\overline{mdim}_{M,\psi}(X,f,d,\varphi)$ is called  $\psi$-induced upper metric mean dimension  with potential $\varphi$ defined in Subsection 2.1.

\end{theorem}

\begin{theorem}\label{thm 1.2}
Let $(X,f)$ be a  TDS  admitting marker property and $\varphi,\psi \in C(X,\mathbb{R})$ with $\psi>0$. Then  for all $d\in \mathscr{D}^{'}(X)$,
\begin{align*}
\overline{{mdim}}_{M,\psi}(X,f,d,\varphi)&=\sup_{\mu\in M(X,f)}\limits \left\{\frac{\underline{rdim}(X,f,d,\mu)}{\int\psi d\mu}+\frac{\int \varphi d\mu}{\int\psi d\mu}\right\}\\
&=\sup_{\mu\in M(X,f)}\limits \left\{\frac{\overline{rdim}(X,f,d,\mu)}{\int\psi d\mu}+\frac{\int \varphi d\mu}{\int\psi d\mu}\right\}.
\end{align*}
\end{theorem}
 
 We would like to emphasize  that  only Theorem 1.2 and subsequent Corollary  3.20 need  the assumption of marker property and hold for some "nice" metrics. 
It is  not clear if we can remove the assumption of marker property in Theorem 1.2. More precisely,  it is unclear  if   for any dynamical system $(X, f)$,  there exists a metric $d\in \mathscr{D}(X)$ such that $mdim(X,f,\varphi)=\overline{{mdim}}_M(X,f,d,\varphi)$. This  open problem was also  mentioned   in \cite{glt16, lt19, t20}.

In 1973,  Bowen \cite{b73} introduced   Bowen  topological entropy  resembling the definition of  Hausdorff dimension for any Borel subset  $Z$ of $X$. In that paper, he proved the following three important results.
\begin{enumerate}[(i)]
 \item  When $Z=X$,  Bowen topological entropy  $h_{top}(f, X)$ coincides with the classical topological entropy.
 \item If $\mu \in M(X,f)$ and $Y\subset X$ with $\mu (Y)=1$, then the measure-theoretic entropy denoted by $h_{\mu}(f)$ is less than  or equal to the Bowen topological entropy   $h_{top}(f,Y)$.
 \item If $\mu \in E(X,f)$,  then the measure-theoretic entropy $h_{\mu}(f)$  is equal to  $h_{top}(f,G_{\mu})$, where  the set $G_\mu=\{x\in X: \lim_{n \to \infty}\frac{1}{n}\sum_{j=0}^{n-1}  \varphi(f^{j}(x))=\int \varphi  d\mu  ~ \mbox{for any }~\varphi \in C(X,\mathbb{R})\}$ denotes the  set of generic points of $\mu$.
\end{enumerate}

 In 2012, Feng and Huang \cite{fh12} introduced  measure-theoretical  upper and lower Brin-Katok local entropies  for  Borel probability measures and obtained variational principles for Bowen topological entropy and Packing topological entropy on  subsets.  Wang and Chen \cite{wc12}  showed the variational principles still holds for BS dimension and Packing BS dimension on subsets.   Following the idea of  the definition of Huasdorff dimension,  Lindenstrauss  and Tuskamoto  \cite{lt19}  introduced mean Hausdorff  dimension, which is proved  to be an upper bound of  mean dimension.  The version of mean Hausdorff dimension with potential  can be found in \cite{t20}. Later, Wang \cite{w21} introduced  Bowen upper metric mean dimension on subsets and established  an an analogous   variational principle for Bowen upper metric mean dimension on subsets.  After that,  Cheng  et al. \cite{cls21} introduced several types of   upper metric mean dimensions  with potential on arbitrary subsets through Carath\rm$\bar {e}$odory-Pesin  structures, which is an analogue of the theory of  topological pressure of non-compact, and they  also  established a  variational principle for Bowen  upper metric mean dimension with potential on  subsets under some conditions.  Inspired by the ideas used in \cite{bs00, wc12}, in this paper  we  introduce the notions of  BS  metric mean dimension and Packing BS metric mean dimension  on subsets, which allows us  to establish  Bowen's equations for Bowen  upper mean dimension and Packing upper metric mean dimension with potential on subsets. Moreover, two   variational principles for BS  metric mean dimension and Packing BS metric mean dimension  on subsets are also obtained analogous to \cite{fh12, wc12,  w21}.  Finally, we  extend Bowen's three important results to the framework of Bowen upper metric mean dimension.

 \begin{theorem}\label{thm 1.3} 
 Let $(X,f)$ be a TDS   with a  metric $d\in \mathscr{D}(X)$ and $Z$ be a non-empty subset of $X$.  Suppose that $\varphi\in C(X,\mathbb{R})$ with $\varphi>0$. Then 
 \begin{enumerate}[(i)]
    \item if $\overline{{mdim}}_{M}(f,X,d)<\infty$, then $\overline{BSmdim}_{M,Z,f}(\varphi)$ is the unique  root  of the equation $\overline{{mdim}}_{M,Z,f}(-t\varphi)=0$;
    \item if $\overline{{Pmdim}}_{M}(f,X,d)<\infty$, then $\overline{BSPmdim}_{M,Z,f}(\varphi)$  is the unique  root of the equation $\overline{{Pmdim}}_{M,Z,f}(-t\varphi)=0$,
 \end{enumerate}	
  where $\overline{{mdim}}_{M,Z,f}(-t\varphi)$, $\overline{{Pmdim}}_{M,Z,f}(-t\varphi)$
  denote Bowen upper metric mean dimension with potential $-t\varphi$ on $Z$ and  Packing upper metric mean dimension with potential $-t\varphi$ on $Z$,  respectively. $\overline{{BSmdim}}_{M,Z,f}(\varphi), \overline{{BSPmdim}}_{M,Z,f}(\varphi)$  are  respectively called BS  metric mean dimension on $Z$ with respect to $\varphi$ and   Packing BS  metric mean dimension on $Z$ with respect to $\varphi$.

  \end{theorem}
        
  \begin{theorem} \label{thm 1.4}
  Let $(X,f)$ be a TDS and $K$ be a non-empty  compact subset of $X$.  Suppose that $\varphi\in C(X,\mathbb{R})$ with $\varphi>0$. Then for all $d\in \mathscr{D}(X)$,
        	
  \begin{align*}
  \overline{{BSmdim}}_{M,K,f}(\varphi,d )&=\limsup_{\epsilon \to 0}\frac{\sup \left\{\underline{h}_{\varphi, \mu}(f,\epsilon): \mu \in M(X), \mu (K)=1\right\}}{\log \frac{1}{\epsilon}},\\
  \overline{{BSPmdim}}_{M,K,f}(\varphi,d )&=\limsup_{\epsilon \to 0}\frac{\sup \left\{\overline{h}_{\varphi, \mu}(f,\epsilon):\mu \in M(X), \mu (K)=1\right\}}{\log \frac{1}{\epsilon}},
  \end{align*}
  where $\underline{h}_{\varphi, \mu}(f,\epsilon)$ and $ \overline{h}_{\varphi, \mu}(f,\epsilon)$ are two notions related to    the measure-theoretical  lower and upper BS  entropies of $\mu$, see  definition \ref{def 3.13}  for  their precise definitions.
  \end{theorem}
        
  \begin{theorem}\label{thm 1.5}
   Let $(X,f)$ be a TDS  with a  metric $d\in \mathscr{D}(X)$, then the following statements hold. 
        	
   \begin{enumerate}[(i)]
   \item Suppose that $\mu \in M(X,f)$. If $Y\subset X$ and $\mu(Y)=1$, then 
   $$\limsup_{\epsilon \to 0} \frac{\underline h_{\mu}^{BK}(f,d,\epsilon)}{\log \frac{1}{\epsilon}}\leq \overline{{mdim}}_{M}(f,Y,d).$$
   \item Suppose that $\mu \in E(X,f)$. If  $\limsup_{\epsilon \to 0}\limits  \frac{\overline h_{\mu}^{BK}(f,d,\epsilon)}{\log \frac{1}{\epsilon}}=\limsup_{\epsilon \to 0}\limits \frac{\underline h_{\mu}^{BK}(f,d,\epsilon)}{\log \frac{1}{\epsilon}}$, then
   \begin{align*}
    \overline{{mdim}}_{M}(f,G_\mu,d)&=\limsup_{\epsilon \to 0}\frac{PS(f,d,\mu,\epsilon)}{\log \frac{1}{\epsilon}}\\
    &=\limsup_{\epsilon \to 0}\frac{\overline h^K_{\mu}(f,d,\epsilon)}{\log \frac{1}{\epsilon}}\\
    &=\limsup_{\epsilon \to 0}\frac{ \overline h_{\mu}^{BK}(f,d,\epsilon)}{\log \frac{1}{\epsilon}}\\
    &=\limsup_{\epsilon \to 0}\frac{ \underline h_{\mu}^{BK}(f,d,\epsilon)}{\log \frac{1}{\epsilon}}=\overline{rdim}_{L^{\infty}}(X,f,d,\mu).
   \end{align*}
        		
   \end{enumerate}	
 All  notions mentioned  in Theorem 1.5 are explicated	 in Subsection \ref{sub 3.4}.
 \end{theorem}

  The rest of this paper is organized as follows. In  section 2, we  introduce the notion of  induced metric mean dimension with potential in subsection 2.1,  and  we prove Theorem 1.1 and Theorem 1.2 in subsection 2.2.   The section 3 is divided into four parts.  In subsection 3.1, we  recall some basic definitions of  upper metric mean dimension with potential  on subsets and collect some standard facts.  Theorem 1.3 and Theorem 1.4 are proved in subsection 3.2 and subsection 3.3, respectively. We give the  proof of  Theorem 1.5  in subsection 3.4.

	%%%%%%%%%%%%%%%%%%%%%%%%%%%%%%%%%%%%%%%%%%%
	
	\section{The   upper metric mean dimension with potential on the whole phase space} 
	 In section  2, we focus on the  upper metric mean dimension with potential on the whole space.  We introduce induced upper metric mean dimension with potential on the whole phase space in subsection 2.1, and we  major the Bowen's equation for upper metric mean dimension with potential on the whole space in subsection 2.2.
	\subsection{Induced upper metric mean dimension with potential}
     In this subsection,  we  present some useful notions associated with upper  metric mean dimension with potential and then  introduce the notion of induced  upper metric mean dimension with potential.
      
     Let $n\in \mathbb{N}$. For $x,y \in X$, we define the $n$-th Bowen metric $d_n$  on $X$ as 
    $$d_n(x,y)=\max_{0\leq j\leq n-1}d(f^{j}(x),f^j(y)).$$
    For each $\epsilon >0$, the \emph{Bowen open  ball and closed ball} of radius $\epsilon$ and  order $n$ in the metric $d_n$ around $x$ are respectively given by 
    $$B_n(x,\epsilon)=\{y\in X: d_n(x,y)<\epsilon\},$$
    $$\overline B_n(x,\epsilon)=\{y\in X:d_n(x,y)\leq\epsilon\}.$$
    For a non-empty subset  $Z\subset X$, one says that a set $E$ is  \emph{an $(n,\epsilon)$-\emph {spanning set} of $Z$} if  for any $x \in Z$, there  exists  $y\in E$ such that $d_n(x,y)<\epsilon.$ The smallest  cardinality  of $(n,\epsilon)$-spanning set of $Z$ is denoted by $r_n(f,d,\epsilon,Z)$. One says that a set $F\subset Z$ is  \emph{an $(n,\epsilon)$-separated set of $Z$} if $d_n(x,y)\geq\epsilon$ for any $x,y \in F$ with $x\not= y$. The maximal  cardinality of  $(n,\epsilon)$-separated set of $Z$    is denoted by $s_n(f,d,\epsilon,Z)$.

    Let $\varphi,\psi \in C(X,\mathbb{R})$ with $\psi>0$. For all  $n \geq1, x\in X$,  we set $S_n\varphi(x):=\sum_{i=0}^{n-1} \limits \varphi(f^{i}(x))$ and  $m:=\min_{x\in X}\limits\psi(x)>0$. 
   We now recall that the equivalent  definition of upper metric mean dimension with potential  defined by  separated set in \cite{t20}.
    
    Let $0<\epsilon<1$, $d\in \mathscr{D}(X)$, and $\varphi \in C(X,\mathbb{R})$. Set
    $$\#_{sep}(X,d_n, S_n\varphi, \epsilon)=\sup\{\sum_{x\in F_n}(1/\epsilon)^{S_n\varphi(x)}: F_n \mbox{ is an}~ (n,\epsilon)\mbox{-separated set  of} ~X\},$$
    and
    $$P(X,f,d,\varphi,\epsilon)=\limsup_{n \to \infty}\frac{\log \#_{sep}(X,d_n, S_n\varphi, \epsilon)}{n}.$$
     \emph{Upper metric mean dimension with potential  $\varphi$} is given by   
    $$\overline{{mdim}}_M(X,f,d,\varphi)=\limsup_{\epsilon \to 0}\frac{P(X,f,d,\varphi,\epsilon)}{\log \frac{1}{\epsilon}}.$$
    
    Specially, $\overline{mdim}_M(X,f,d)=\overline{{mdim}}_M(X,f,d,0)$ recovers the definition of   the upper metric mean dimension of $X$ given by Lindenstrauss and Weiss in  \cite{lw00}. 

\begin{definition}
	Let $(X,f)$ be a  TDS  with a  metric $d \in \mathscr{D}(X)$ and $\varphi,\psi \in C(X,\mathbb{R})$ with $\psi>0$.  For $T>0$, set\\
	$$	S_T:=\{n\in \mathbb{N}: \exists x \in X ~\mbox{such that } S_n\psi(x)\leq T~ \mbox{and}~ S_{n+1}\psi(x)>T\}.$$
	For  each $n\in S_T$ and   $\epsilon >0$, put
	$$X_n=\{x\in X:  S_n\psi(x)\leq T~ \mbox{and}~ S_{n+1}\psi(x)>T\},$$
	\begin{align*}
		P_{\psi, T}(X,f,d,\varphi,\epsilon)
		= \sup\left\{\sum_{n\in S_T}\sum_{x \in F_n}\limits(1/\epsilon)^ {S_n\varphi(x)}: F_n \mbox{ is an}~ (n,\epsilon)\mbox{-separated set  of} ~X_n, n \in S_T \right\}.
	\end{align*}
	We define the \emph{$\psi$-induced  upper metric mean dimension with potential $\varphi$} as  $$\overline{{mdim}}_{M,\psi}(X,f,d,\varphi)=\limsup_{\epsilon \to 0}\limsup_{T \to \infty}\frac{1}{T\log{\frac{1}{\epsilon}}}\log P_{\psi, T}(X,f,d,\varphi,\epsilon).$$
\end{definition}                                         
\begin{remark}
	\begin {enumerate}[(i)]
	\item  If $S_T\not=\emptyset$, then for each $n\in S_T$,  we have $n\leq [\frac{T}{m}] +1$, where $ [\frac{T}{m}] $ denotes the integer part of $ \frac{T}{m}$ and $m=\min_{x\in X}\limits\psi(x)$. In other words,  $S_T$ is a finite set.
	\item Take $\psi =1$, then the $\psi$-induced  upper metric mean dimension with potential $\varphi$ is reduced to the upper metric  mean dimension with potential $\varphi$, that is, $\overline{{mdim}}_{M,1}(X,f,d,\\
	\varphi)=\overline{{mdim}}_M(X,f,d,\varphi).$
	\item $ \overline{{mdim}}_{M,\psi}(X,f,d,\varphi) >-\infty.$
\end{enumerate}
\end{remark}

In fact, analogous to the definition of  the  classical topological pressure, the  $\psi$-induced  upper metric mean dimension with potential $\varphi$ can be  also  given by spanning set.
\begin{proposition}\label{poro1.4}
	Let $(X,f)$ be a  TDS  with a  metric $d \in \mathscr{D}(X)$ and $\varphi,\psi \in C(X,\mathbb{R})$ with $\psi>0$. 
	Set
	$$Q_{\psi, T}(X,f,d,\varphi,\epsilon)
	= \inf\left\{\sum_{n\in S_T}\sum_{x \in E_n}\limits(1/\epsilon)^ {S_n\varphi(x)}: E_n \mbox{ is an}~ (n,\epsilon)\mbox{-spanning set  of} ~X_n, n \in S_T \right\}.$$
	Then $$\overline{{mdim}}_{M,\psi}(X,f,d,\varphi)=\limsup_{\epsilon \to 0}\limsup_{T \to \infty}\frac{1}{T\log{\frac{1}{\epsilon}}}\log Q_{\psi, T}(X,f,d,\varphi,\epsilon).$$ 
\end{proposition}

\begin{proof}
	Let $0<\epsilon<1$,  $n\in S_T$. Note that a  maximal $(n,\epsilon)$-separated set $F_n$  of $X_n$ is also an $(n,\epsilon)$-spanning set of $X_n$. Then 
	$$Q_{\psi, T}(X,f,d,\varphi,\epsilon)\leq P_{\psi, T}(X,f,d,\varphi,\epsilon).$$
	Therefore,
	$$\limsup_{\epsilon \to 0}\limsup_{T \to \infty}\frac{1}{T\log{\frac{1}{\epsilon}}}\log Q_{\psi, T}(X,f,d,\varphi,\epsilon)\leq \overline{{mdim}}_{M,\psi}(X,f,d,\varphi).$$
	
	On the other hand, let $0<\epsilon<1$ and $\gamma(\epsilon):=\sup\{|\varphi(x)-\varphi(y)|:d(x,y)<\epsilon\}$. For $n\in S_T$, let $E_n$ be an $(n,\frac{\epsilon}{2})$-spanning set of $X_n$ and  $F_n$ be an $(n,\epsilon)$-separated set of $X_n$. Consider a map $\Phi: F_n \rightarrow E_n$ by assigning each $x\in F_n$ to $\Phi(x)\in E_n$ satisfying $d_n(x,\Phi(x))<\frac{\epsilon}{2}$. Then $\Phi$ is injective.
	
	Thus 
	\begin{align*}
		&\sum_{n\in S_T}\sum_{y \in E_n}\limits(2/\epsilon)^ {S_n\varphi(y)}\\
		\geq&\sum_{n\in S_T}\sum_{x \in F_n}\limits(2/\epsilon)^ {S_n\varphi(\Phi(x))}\\
		=&\sum_{n\in S_T}\sum_{x \in F_n}\limits(2/\epsilon)^ {S_n\varphi(\Phi(x))-S_n\varphi(x)+S_n\varphi(x)}\\
		\geq&(2/\epsilon)^{-\gamma(\epsilon)(\frac{T}{m}+1)}\sum_{n\in S_T}\sum_{x \in F_n}\limits(1/\epsilon)^ {S_n\varphi(x)}\cdot2^{S_n \varphi(x)}\\
		\geq&(2/\epsilon)^{-\gamma(\epsilon)(\frac{T}{m}+1)}\cdot2^{-(\frac{T}{m}+1)||\varphi||}\sum_{n\in S_T}\sum_{x \in F_n}\limits(1/\epsilon)^ {S_n\varphi(x)}
	\end{align*}
	It follows that  
	$$\limsup_{T \to \infty}\frac{1}{T}\log Q_{\psi, T}(X,f,d,\varphi,\frac{\epsilon}{2})\geq -\frac{1}{m} \gamma(\epsilon)\log \frac{2}{\epsilon}-\frac{||\varphi||}{m}\log2 +\limsup_{T \to \infty}\frac{1}{T}\log P_{\psi, T}(X,f,d,\varphi,\epsilon) .$$
	Since $\epsilon \to 0$, $\gamma(\epsilon) \to 0$, we finally deduce that 
	$$\limsup_{\epsilon \to 0}\limsup_{T \to \infty}\frac{1}{T\log{\frac{1}{\epsilon}}}\log Q_{\psi, T}(X,f,d,\varphi,\epsilon)\geq \overline{{mdim}}_{M,\psi}(X,f,d,\varphi).$$
	
\end{proof}

	\subsection{Bowen's equation for upper metric mean dimension with potential on the whole phase space}
	In this subsection, we prove Theorem 1.1 and Theorem 1.2. To this end, we  need to   examine the relationship between $\overline{{mdim}}_M(X,f,d,\varphi)$ and  $\overline{{mdim}}_{M,\psi}(X,f,d,\varphi)$, which will be useful  for the  forthcoming proof.
	
	\begin{theorem}\label{thm 2.4}
		 Let $(X,f)$ be a  TDS with a  metric $d \in \mathscr{D}(X)$   and $\varphi,\psi \in C(X,\mathbb{R})$ with $\psi>0$.  For $T>0$, define\\
		$$G_T:=\{n\in \mathbb{N}: \exists x \in X ~\mbox{such that } S_n\psi(x)>T\}.$$
		For  each $n\in G_T$ and   $\epsilon >0$, define
		$$Y_n=\{x\in X: S_n\psi(x)>T\},$$
		\begin{align*}
			R_{\psi, T}(X,f,d,\varphi,\epsilon)
			= \sup\left\{\sum_{n\in G_T}\sum_{x \in F_n^{'}}\limits(1/\epsilon)^ {S_n\varphi(x)}:F_n ^{'}\mbox{ is an}~ (n,\epsilon)\mbox{-separated set  of} ~Y_n, n \in G_T \right\}.
		\end{align*}
	Then
	\begin{align}\label{inequ 2.1}
			\overline{{mdim}}_{M,\psi}(X,f,d,\varphi)=\inf\{\beta \in \mathbb{R}: \limsup_{\epsilon \to 0}\limsup_{T \to \infty}R_{\psi, T}(X,f,d,\varphi-\beta \psi,\epsilon)<\infty\}.
	\end{align}

	We use the convention that $\inf\emptyset=\infty$.
  \end{theorem}

\begin{proof}
	For $n\in \mathbb{N}, x\in X$,  we define $m_n(x)$ as the unique positive integer  satisfying that 
	\begin{align}\label{inequ 2.2}
		(m_n(x)-1)||\psi||<S_n\psi(x)\leq m_n(x)||\psi||.
	\end{align}
	   For any $x\in X$, we have
	\begin{align}\label{inequ 2.3}
		(1/\epsilon)^{-\beta ||\psi||m_n(x)}(1/\epsilon)^{-|\beta|||\psi||}\leq (1/\epsilon)^{-\beta S_n\psi(x)}\leq (1/\epsilon)^{-\beta ||\psi||m_n(x)}(1/\epsilon)^{|\beta|||\psi||}
	\end{align}
	for all $\beta \in\mathbb{R}$.
	
  Fix $0<\epsilon<1$. 	Define 
  \begin{align*}
	&R_{\psi, T}^{(1)}(X,f,d,\varphi,\{\beta||\psi||m_n+|\beta|||\psi||\}_{n\in G_T},\epsilon)
	=\\
	 &\sup\left\{\sum_{n\in G_T}\sum_{x \in F_n^{'}}\limits(1/\epsilon)^ {S_n\varphi(x)-\beta||\psi||m_n(x)-|\beta|||\psi||}: F_n ^{'}\mbox{ is an}~ (n,\epsilon)\mbox{-separated set  of} ~Y_n, n \in G_T \right\},
  \end{align*}	
\begin{align*}
	&R_{\psi, T}^{(2)}(X,f,d,\varphi,\{\beta||\psi||m_n-|\beta|||\psi||\}_{n\in G_T},\epsilon)
	=\\
	&\sup\left\{\sum_{n\in G_T}\sum_{x \in F_n^{'}}\limits(1/\epsilon)^ {S_n\varphi(x)-\beta||\psi||m_n(x)+|\beta|||\psi||}:F_n ^{'}\mbox{ is an}~ (n,\epsilon)\mbox{-separated set  of} ~Y_n, n \in G_T \right\}.
\end{align*}
 Set
 \begin{align*}
 	A&=\inf\{\beta \in \mathbb{R}:\limsup_{\epsilon \to 0}\limsup_{T \to \infty}R_{\psi, T}^{(1)}(X,f,d,\varphi,\{\beta||\psi||m_n+|\beta|||\psi||\}_{n\in G_T},\epsilon)<\infty\},\\
 	B&=\inf\{\beta \in \mathbb{R}: \limsup_{\epsilon \to 0}\limsup_{T \to \infty}R_{\psi, T}(X,f,d,\varphi-\beta \psi,\epsilon)<\infty\},\\
 	C&=\inf\{\beta \in \mathbb{R}:\limsup_{\epsilon \to 0}\limsup_{T \to \infty}R_{\psi, T}^{(2)}(X,f,d,\varphi,\{\beta||\psi||m_n-|\beta|||\psi||\}_{n\in G_T},\epsilon)<\infty\}.
 	\end{align*}
	By  (\ref{inequ 2.3}), we have 
	$A\leq B\leq C$. To  get (\ref{inequ 2.1}), it suffices to show
	$$\overline{{mdim}}_{M,\psi}(X,f,d,\varphi)\leq A,  ~\mbox{and} ~~C\leq \overline{{mdim}}_{M,\psi}(X,f,d,\varphi).$$
	
	 Firstly, we  show that $\overline{{mdim}}_{M,\psi}(X,f,d,\varphi)\leq A$. Let  $\beta<\overline{{mdim}}_{M,\psi}(X,f,d,\varphi)$,  we can choose a positive number $\delta>0$ and  a sequence of positive number $0<\epsilon_k<1$ such that  $$\beta+\delta<\overline{{mdim}}_{M,\psi}(X,f,d,\varphi),$$
	 and 
	 $$\overline{{mdim}}_{M,\psi}(X,f,d,\varphi)=\lim\limits_{k \to \infty}\limsup_{T \to \infty}\frac{1}{\log(1/\epsilon_k)T}\log P_{\psi,T}(X,f,d,\varphi,\epsilon_k).$$
	 
	Hence, there exists $K_0\in \mathbb{N}$ such that for any $k>K_0$, we can choose a subsequence $\{T_j\}_{j\in \mathbb{N}}$  convergences to  $\infty$ as $j \to \infty$ satisfying  that  
	$$(1/\epsilon_k)^{T_j(\beta+\frac{\delta}{2})}<P_{\psi,T_j}(X,f,d,\varphi,\epsilon_k).$$
	
	For every $j\in \mathbb{N}$ and $n\in S_{T_j}$ there is  an $(n,\epsilon)$-separated set $F_n$ of $X_n$  so that
	\begin{align}\label{inequ 2.4}
		(1/\epsilon_k)^{T_j(\beta+\frac{\delta}{2})}<\sum_{n\in S_{T_j}}\sum_{x\in F_n}(1/\epsilon_k)^{S_n\varphi(x)}.
	\end{align}

	 Claim 1: Let $T>0$. If $S_T\not=\emptyset$, then for each $n\in S_T$, we have  $$\frac{T}{||\psi||}-1<n\leq [\frac{T}{m}]+1,$$
	 where $m=\min_{x\in X}\psi(x)>0$.
	 
	 Proof of the Claim 1: Let $n\in S_T$, then there exists a point $x\in X$ such that $S_n\psi(x)\leq T$ and $S_{n+1}\psi(x)>T$. It follows that $T-||\psi||<S_n\psi(x)\leq T$, which implies  the desired claim.

	 Taking $T_{j_1}$ arbitrarily,  note that $T_j \to \infty$, then we  can choose  $T_{j_2}$ such that 
	 $$[\frac{T_{j_1}}{m}]+1<\frac{T_{j_2}}{||\psi||}-1.$$
	 
	Repeating this process, we can choose a subsequence $T_{j_k}$ of $T_j$ that   convergences to $\infty$ as $k\to \infty$. Without loss of generality, we still denote  the subsequence $T_{j_k}$ by $T_j$. 
	
	Claim 2: $S_{T_i}\cap S_{T_j}=\emptyset$ with $i\not= j$.
	
	Proof of the Claim 2:  If $S_{T_i}\cap S_{T_j}\not=\emptyset$, we assume that $i<j$ and
	let $n\in  S_{T_i}\cap S_{T_j}$,
	then there exists $x_1,x_2\in X$ such that $S_n\psi(x_1)\leq T_i$ and $S_{n+1}\psi(x_1)>T_i$, 
	$S_n\psi(x_2)\leq T_j$ and $S_{n+1}\psi(x_2)>T_j$.
	By Claim 1, we have $$n\leq [\frac{T_i}{m}]+1 <\frac{T_j}{||\psi||}-1<n,$$
	 so we get a  contradiction.

	Note that for each $j\in \mathbb{N}$ and $n\in S_{T_j}$, if $x\in F_n$, then we have $T_j-||\psi||<S_n\psi(x)\leq T_j$. Together with inequality (\ref{inequ 2.2}), we get 
	  \begin{align}\label{inequ 2.5}
	  	|||\psi||m_n(x)-T_j|<2||\psi||.
	  \end{align}
  Observed that $-\beta ||\psi||m_n(x)\geq -\beta T_j -2|\beta|||\psi||.$ 
  This gives us 
	  \begin{align*}
	  	&R_{\psi, T}^{(1)}(X,f,d,\varphi,\{\beta||\psi||m_n+|\beta|||\psi||\}_{n\in G_T},\epsilon_k)\\
	  	&\geq \sum_{j\in \mathbb{N},T_j-||\psi||>T}\sum_{n\in S_{T_j}}\sum_{x\in F_n}(1/\epsilon_k)^{S_n\varphi(x)-\beta||\psi||m_n(x)-|\beta|||\psi||}\\  
	  	&\geq (1/\epsilon_k)^{-3|\beta|||\psi||}\sum_{j\in \mathbb{N},T_j-||\psi||>T}\sum_{n\in S_{T_j}}\sum_{x\in F_n}(1/\epsilon_k)^{S_n\varphi(x)-\beta T_j}\\	
	  	&\geq (1/\epsilon_k)^{-3|\beta|||\psi||}\sum_{j\in \mathbb{N},T_j-||\psi||>T}(1/\epsilon_k)^{\frac{\delta}{2} T_j}  ~~~~~~ \mbox{by} ~(\ref{inequ 2.4})\\
	  	&=\infty.
	  \end{align*}
   It follows that  for any $\beta<\overline{{mdim}}_{M,\psi}(X,f,d,\varphi)$, we have 
   \begin{align}\label{inequ 2.6}
   	\limsup_{\epsilon \to 0}\limsup_{T \to \infty}R_{\psi, T}^{(1)}(X,f,d,\varphi,\{\beta||\psi||m_n+|\beta|||\psi||\}_{n\in G_T},\epsilon)=\infty.
   \end{align}
 Therefore,  we obtain $\overline{{mdim}}_{M,\psi}(X,f,d,\varphi)\leq A.$
 
  If $\overline{{mdim}}_{M,\psi}(X,f,d,\varphi)=\infty$, let $P<\overline{{mdim}}_{M,\psi}(X,f,d,\varphi)$, by   slightly modifying the above proof,
   one can show  for any $\beta <P$,
   \begin{align*}
   	\limsup_{\epsilon \to 0}\limsup_{T \to \infty}R_{\psi, T}^{(1)}(X,f,d,\varphi,\{\beta||\psi||m_n+|\beta|||\psi||\}_{n\in G_T},\epsilon)=\infty.
   \end{align*}
Since $P$ is arbitrary, and using the convention, we know that  $A=\inf \emptyset=\infty.$  Hence  $$\overline{{mdim}}_{M,\psi}(X,f,d,\varphi)=A=\infty.$$

   Now, we turn to  show $C\leq \overline{{mdim}}_{M,\psi}(X,f,d,\varphi).$   We assume that $\overline{{mdim}}_{M,\psi}(X,f,d,\varphi)<\infty$, otherwise  there is nothing to prove.  Let $\delta >0$, using the definition of $\overline{{mdim}}_{M,\psi}(X,f,d,\varphi)$, there   is a $0<\epsilon_0<1$ so that for  every $0<\epsilon<\epsilon_0$,  
   $$\limsup_{T \to \infty}\frac{1}{\log(1/\epsilon)T}\log P_{\psi,T}(X,f,d,\varphi,\epsilon)<\overline{{mdim}}_{M,\psi}(X,f,d,\varphi)+\frac{\delta}{2}.$$
   Hence,  we can choose an $l_0\in \mathbb{N}$ such that   for  any $l\geq l_0$, 
   \begin{align}\label{inequ 2.7}
   	&P_{\psi,lm}(X,f,d,\varphi,\epsilon)<(1/\epsilon)^{lm(\overline{{mdim}}_{M,\psi}(X,f,d,\varphi)+\frac{2\delta}{3})},
   \end{align}
   \begin{align} \label{inequ 2.8}
   	\frac{\delta}{3}l_0m-\Delta-1>0,
   \end{align}
   where $\Delta=3||\psi||(|\overline{{mdim}}_{M,\psi}(X,f,d,\varphi)+\delta|)$ and  $m=\min_{x\in X}\limits \psi(x).$ \\
   For  $n\in S_{lm}$, let $F_n$  be an $(n,\epsilon)$-separated set of $X_n$. Then for each $x\in F_n$, we have
   $$	|||\psi||m_n(x)-lm|<2||\psi||.$$
 Therefore,  we  obtain that
   \begin{align} \label{inequ 2.9}
   	&-(\overline{{mdim}}_{M,\psi}(X,f,d,\varphi)+\delta)||\psi||m_n(x)  \nonumber \\ 
   	&\leq -lm(\overline{{mdim}}_{M,\psi}(X,f,d,\varphi)+\delta)+2||\psi||(|\overline{{mdim}}_{M,\psi}(X,f,d,\varphi)+\delta|). 
   \end{align}

   For sufficiently large $T$,  $n\in G_T$, let $F_n^{'}$  be an $(n,\epsilon)$-separated set of $Y_n$. Then for each $x\in F_n^{'}$, there exists  a unique $l\geq l_0$ such that $(l-1)m<S_n\psi(x)\leq lm$. So $S_{n+1}\psi(x)=S_n\psi(x)+\psi(f^nx)>lm$. It follows that 
   \begin{align*}
   	&R_{\psi, T}^{(2)}(X,f,d,\varphi,\{(\overline{{mdim}}_{M,\psi}(X,f,d,\varphi)+\delta)||\psi||m_n-|(\overline{{mdim}}_{M,\psi}(X,f,d,\varphi)+\delta)|||\psi||\}_{n\in G_T},\epsilon)\\
   	&\leq \sum_{l\geq l_0}\sup\left\{\sum_{n\in S_{lm}}\sum_{x\in F_n}(1/\epsilon)^{S_n\varphi(x)-(\overline{{mdim}}_{M,\psi}(X,f,d,\varphi)+\delta)||\psi||m_n(x)+|(\overline{{mdim}}_{M,\psi}(X,f,d,\varphi)+\delta)|||\psi||},\right. \\
   	&~~~~~~~~~~~~~~~~~~~~~~~~~~~~~~~~~~~~~~\left. F_n ~\mbox{is an}~ (n,\epsilon)\mbox{-separated set  of} ~X_n, n \in S_{lm} \right\}\\
   	&\leq(1/\epsilon)^{3|(\overline{{mdim}}_{M,\psi}(X,f,d,\varphi)+\delta)|||\psi||}\sum_{l\geq l_0}\sup\left\{\sum_{n\in S_{lm}}\sum_{x\in F_n}(1/\epsilon)^{S_n\varphi(x)-lm(\overline{{mdim}}_{M,\psi}(X,f,d,\varphi)+\delta)} ,\right. \\
   	& \left.~~~~~~~~~~~~~~~~~~~~~~~~~~~~~~~~~~~~~F_n ~\mbox{is an}~ (n,\epsilon)\mbox{-separated set  of} ~X_n, n \in S_{lm} \right\}~~~ \mbox{by} ~~(\ref{inequ 2.9})\\
   	&\leq (1/\epsilon)^{\Delta}\sum_{l\geq l_0}(1/\epsilon)^{-\frac{\delta}{3}lm}  ~~~ \mbox{by}~~ (\ref{inequ 2.7})\\
   	&\leq\frac{{\epsilon}^{\frac{\delta}{3}l_0m-\Delta} }{1-{\epsilon}^{\frac{\delta}{3}m}}<\frac{\epsilon} {1-{\epsilon}^{\frac{\delta}{3}m}}< \frac{1} {1-{\epsilon}^{\frac{\delta}{3}m}}~~~ \mbox{by} ~~(\ref{inequ 2.8}).
   \end{align*}
   
  Therefore, we get 
   \begin{align*}
   	&\limsup_{\epsilon \to 0}\limsup_{T \to \infty}\log R_{\psi, T}^{(2)}(X,f,d,\varphi,\\ &\{(\overline{{mdim}}_{M,\psi}(X,f,d,\varphi)+\delta)||\psi||m_n-|(\overline{{mdim}}_{M,\psi}(X,f,d,\varphi)+\delta)|||\psi||\}_{n\in G_T},\epsilon)\leq 1.
   \end{align*}
   That is to say, $C\leq \overline{mdim}_{M,\psi}(X,f,d,\varphi)+\delta $, 
   and hence we  obtain $C\leq \overline{mdim}_{M,\psi}(X,f,d,\varphi)$ by letting $\delta \to 0$.

\end{proof}

\begin{corollary}\label{coro 2.5}
	Let $(X,f)$ be a  TDS  with a metric $d \in \mathscr{D}(X)$ and $\varphi,\psi \in C(X,\mathbb{R})$ with $\psi>0$. Then
	$$	\overline{mdim}_{M,\psi}(X,f,d,\varphi)\geq\inf\{\beta\in \mathbb{R}:\overline{mdim}_{M}(X,f,d,\varphi-\beta\psi)\leq0\}.$$	
\end{corollary}

\begin{proof}
	Let $\beta \in \{\beta \in \mathbb{R}: \limsup_{\epsilon \to 0}\limits \limsup_{T \to \infty}\limits R_{\psi, T}(X,f,d,\varphi-\beta \psi,\epsilon)<\infty\},$ and let $M:=\limsup_{\epsilon \to 0}\limits \limsup_{T \to \infty}\limits R_{\psi, T}(X,f,d,\varphi-\beta \psi,\epsilon)$. Then we can  find $0<\epsilon_0<1$ such that for any $0<\epsilon<\epsilon_0$,  there is a $T_0\in \mathbb{N}$  so that for all $T\geq T_0$, we have 
	$$R_{\psi, T}(X,f,d,\varphi-\beta \psi,\epsilon)<M+1.$$
	There exists a subsequence $n_j$   that convergences to  $\infty$ as $j \to \infty$ such that
	\begin{align*}
		P(X,f,d,\varphi-\beta\psi,\epsilon)&=\limsup_{n \to \infty}\frac{\log \#_{sep}(X,d_n, S_n(\varphi-\beta\psi), \epsilon)}{n}\\
		&=\lim_{j \to \infty}\frac{\log \#_{sep}(X,d_{n_j}, S_{n_j}(\varphi-\beta\psi), \epsilon)}{n_j}.
	\end{align*} 
	Therefore,  for each  $T\geq T_0$, there exists  sufficiently   large positive number $n_j>T$ such that 
	$S_{n_j}\psi(x)>T$ for all $x\in X$.  Hence,  $n_j \in G_T$. Let $F_{n_j}$ be an  $(n_j,\epsilon)$-separated set  of $X$. Then
	$$\sum_{x\in F_{n_j}}\limits (1/\epsilon)^{S_{n_j}(\varphi(x)-\beta\psi(x))}<M+1,$$
	which shows that $P(X,f,d,\varphi-\beta\psi,\epsilon)\leq0$.
	This yields that $\overline{mdim}_{M}(X,f,d,\varphi-\beta\psi)\leq0$.
	By Theorem \ref{thm 2.4}, we deduce that 
	
	$$\overline{{mdim}}_{M,\psi}(X,f,d,\varphi)
	\geq\inf\{\beta\in \mathbb{R}:\overline{mdim}_{M}(X,f,d,\varphi-\beta\psi)\leq0\}.$$	
\end{proof}

The following proposition describes  some  properties of the function  $\overline{{mdim}}_M(X,f,d,\varphi-\beta\psi)$ with respect to $\beta$,
which is useful for establishing the Bowen's equation for upper metric mean dimension with potential on the whole phase space. 

 \begin{proposition}\label{prop 2.6}
	Let $(X,f)$ be a  TDS with a metric $d \in \mathscr{D}(X)$  and $\varphi,\psi \in C(X,\mathbb{R})$ with $\psi>0$. Consider the map  $\beta \in \mathbb{R} \longmapsto \overline{mdim}_{M}(X,f,d,\varphi-\beta\psi)$. The following statements hold. 
	\begin{enumerate}[(i)]
		\item If  $\overline{mdim}_{M}(X,f,d,\varphi-\beta_0\psi)=\infty$ for some $ \beta_0 \in \mathbb{R}$, then the map $\overline{mdim}_{M}(X,f,d,\varphi-\cdot\psi)$ is infinite.
		\item If  $\overline{mdim}_{M}(X,f,d,\varphi-\beta_0\psi)<\infty$ for some $ \beta_0 \in \mathbb{R}$, then the map $\overline{mdim}_{M}(X,f,d,\varphi-\cdot\psi)$ is finite,  strictly  decreasing and  continuous on $\mathbb{R}$. Moreover, the equation  $\overline{mdim}_{M}(X,f,d,\varphi-\beta\psi)=0$ has unique (finite) root. 
	\end{enumerate}	
\end{proposition}

\begin{proof}
	Given  $0<\epsilon<1$. For $\beta_1,\beta_2\in \mathbb{R}$ and each $n\in \mathbb{N}$,  
	\begin{align*}
		&\sum_{x\in E} (1/\epsilon)^{S_n \varphi(x)-\beta_2 S_n\psi(x)-n|\beta_1-\beta_2|\cdot||\psi||}\\
		\leq&\sum_{x\in E} (1/\epsilon)^{S_n\varphi(x)-\beta_1 S_n\psi(x)}\\
		\leq&\sum_{x\in E} (1/\epsilon)^{S_n\varphi(x)-\beta_2 S_n\psi(x)+n|\beta_1-\beta_2|\cdot||\psi||},
	\end{align*}
	where $E$  is an $(n,\epsilon)$-separated set of $X$.
	
	Therefore,  
	\begin{align} \label{inequ 2.10}
		\overline{{mdim}}_{M}(X,f,d,\varphi-\beta_2\psi)-|\beta_1-\beta_2|||\psi||&\leq \overline{{mdim}}_{M}(X,f,d,\varphi-\beta_1\psi)\nonumber\\
		&\leq \overline{{mdim}}_{M}(X,f,d,\varphi-\beta_2\psi)+|\beta_1-\beta_2|||\psi||.
	\end{align}
	This yields that $	\overline{{mdim}}_{M}(X,f,d,\varphi-\beta_1\psi)<\infty$ if and only if  $	\overline{{mdim}}_{M}(X,f,d,\varphi-\beta_2\psi)<\infty$, which confirms our corresponding statements.
	
	Under the assumption of (ii), we prove the remaining statements.
	
	It follows from the inequality (\ref{inequ 2.10}) that 
	$$|\overline{mdim}_{M}(X,f,d,\varphi-\beta_1\psi)-\overline{mdim}_{M}(X,f,d,\varphi-\beta_2\psi)|\leq||\psi|||\beta_1-\beta_2|.$$
	This tells us the map $\overline{mdim}_{M}(X,f,d,\varphi-\cdot\psi)$ is continuous on  $\mathbb{R}$.
	
	Let  $\beta_1, \beta_2 \in \mathbb{R}$  with $\beta_1<\beta_2$,  and fix $0<\epsilon<1$. Let $F_n$ be an $(n,\epsilon)$-separated set of $X$, we have
	\begin{align*}
		&\sum_{x\in F_n}(1/\epsilon)^{S_n \varphi(x)-\beta_2S_n \psi(x)}\\
		=&\sum_{x\in F_n}(1/\epsilon)^{S_n \varphi(x)-\beta_1S_n \psi(x)+(\beta_1-\beta_2)S_n\psi(x)}	\\
		\leq& \sum_{x\in F_n}(1/\epsilon)^{S_n \varphi(x)-\beta_1S_n 
			\psi(x)+(\beta_1-\beta_2)nm},
	\end{align*}
	where  $m=\min_{x\in X}\psi(x)>0$.
	
	Then we obtain that 
	\begin{align}\label{inequ 2.11}
		\overline{mdim}_{M}(X,f,d,\varphi-\beta_2\psi)\leq\overline{mdim}_{M}(X ,f,d,\varphi-\beta_1\psi)-(\beta_2-\beta_1)m,
	\end{align}
	
	which implies that  the map $\overline{mdim}_{M}(X,f,d,\varphi-\cdot\psi)$ is strictly  decreasing.

	If $\overline{mdim}_{M}(X,f,d,\varphi)=0$, then $0$ is the unique root of the equation  $\overline{mdim}_{M}(X,f,d,\varphi-\beta\psi)=0$. 
	
	If $\overline{mdim}_{M}(X,f,d,\varphi)\not=0$, we assume  that $\overline{mdim}_{M}(X,f,d,\varphi)>0$,  taking $\beta_1=0$  and $\beta_2=h>0$ in $(\ref{inequ 2.11})$,
	then 
	$$\overline{mdim}_{M}(X,f,d,\varphi-h\psi)\leq\overline{mdim}_{M}(X ,f,d,\varphi)-hm.$$
	
	Hence, the unique root  $\beta$ of the equation $\overline{mdim}_{M}(X,f,d,\varphi-\beta\psi)=0$  satisfies  $0<\beta \leq \frac{\overline{mdim}_{M}(X, f, d, \varphi)}{m}.$
	
	For the case $\overline{mdim}_{M}(X,f,d,\varphi)<0$,  taking $\beta_1=h<0$ and $\beta_2=0$ in $(\ref{inequ 2.11})$ again, 
	$$\overline{mdim}_{M}(X,f,d,\varphi)-hm\leq \overline{mdim}_{M}(X,f,d,\varphi-h\psi).$$
	
	Then  the unique root  $\beta$ of the equation   $\overline{mdim}_{M}(X,f,d,\varphi-\beta\psi)=0$ satisfies  $\frac{\overline{mdim}_{M}(X ,f,d,\varphi)}{m}\leq \beta <0$. 	
	
\end{proof}

\begin{corollary}\label{coro 2.7}
	Let $(X,f)$ be a  TDS   with a metric $d \in \mathscr{D}(X)$   and $\varphi,\psi \in C(X,\mathbb{R})$ with $\psi>0$.  Then
	\begin{align*}
		\overline{mdim}_{M,\psi}(X,f,d,\varphi)&=\inf\{\beta\in \mathbb{R}:\overline{mdim}_{M}(X,f,d,\varphi-\beta\psi)\leq0\}\\
		&=\sup\{\beta\in \mathbb{R}:\overline{mdim}_{M}(X,f,d,\varphi-\beta\psi)\geq0\}.
	\end{align*}	
\end{corollary}	

\begin{proof}

If there exists $\beta_0\in \mathbb{R}$ such that $\overline{mdim}_{M}(X,f,d,\varphi-\beta_0\psi)=\infty$, then by Proposition \ref{prop 2.6}, $\overline{mdim}_{M}(X,f,d,
\varphi-\beta\psi)=\infty$ for all $\beta \in \mathbb{R}$.  Using  Corollary \ref{coro 2.5}, we obtain that 
	\begin{align*}
	\overline{mdim}_{M,\psi}(X,f,d,\varphi)
	&=\sup\{\beta\in \mathbb{R}:\overline{mdim}_{M}(X,f,d,\varphi-\beta\psi)\geq0\}\\
	&=\inf\{\beta\in \mathbb{R}:\overline{mdim}_{M}(X,f,d,\varphi-\beta\psi)\leq0\}
	=\inf \emptyset =\infty	.
\end{align*}

	Now, we can  assume that $\overline{mdim}_{M}(X,f,d,\varphi-\beta\psi)\in\mathbb{R}$ for any $\beta \in \mathbb{R}$. 
	
	Next, we show that
	\begin{align}\label{inequ 2.12}
		\overline{mdim}_{M,\psi}(X,f,d,\varphi)\leq\inf\{\beta\in \mathbb{R}:\overline{mdim}_{M}(X,f,d,\varphi-\beta\psi)<0\}.
	\end{align} 	
   Let $\beta \in \mathbb{R}$ with $\overline{mdim}_{M}(X,f,d,\varphi-\beta \psi)=2a<0.$
  Then there exists $0<\epsilon_0<1$ such that for any $0<\epsilon<\epsilon_0$, we can choose   $N_0$ such that for $n\geq N_0$, one has 
  $$\sup\left\{\sum_{x\in F_n}\limits (1/\epsilon)^{S_n(\varphi(x)-\beta\psi(x))}: F_n ~\mbox{is an } (n,\epsilon)\mbox{-separated set  of }X\right\}<(1/\epsilon)^{an}.$$
This  implies that  for sufficiently  large $T$, we have 
 \begin{align*}
	R_{\psi, T}(X,f,d,\varphi-\beta \psi,\epsilon)&\leq\sum_{n\geq N_0}\sup_{ F_n}\sum_{x\in F_n}\limits (1/\epsilon)^{S_n(\varphi(x)-\beta\psi(x))}\\
	&\leq \sum_{n\geq N_0}(1/\epsilon)^{an}\\
	&<\frac{1}{1-\epsilon^{-a}}.
\end{align*}
  We finally obtain that $\limsup_{\epsilon \to 0}\limits \limsup_{T \to \infty}\limits R_{\psi, T}(X,f,d,\varphi-\beta \psi,\epsilon)\leq 1$. It follows from Theorem \ref{thm 2.4} that 
    \begin{align*}
	&\inf\{\beta\in \mathbb{R}:\overline{mdim}_{M}(X,f,d,\varphi-\beta \psi)<0\}\\
	\geq &\inf\{\beta\in \mathbb{R}:\limsup_{\epsilon \to 0}\limits \limsup_{T \to \infty}\limits R_{\psi, T}(X,f,d,\varphi-\beta \psi,\epsilon)< \infty\}\\
	=&\overline{mdim}_{M,\psi}(X,f,d,\varphi).
\end{align*}

By virtue of  Proposition \ref{prop 2.6}, we know that 
\begin{align}\label{inequ 2.13}
	&\inf\{\beta\in \mathbb{R}:\overline{mdim}_{M}(X,f,d,\varphi-\beta \psi)<0\}\nonumber\\
	=&\inf\{\beta\in \mathbb{R}:\overline{mdim}_{M}(X,f,d,\varphi-\beta \psi)\leq0\}\nonumber\\
	=&\sup\{\beta\in \mathbb{R}:\overline{mdim}_{M}(X,f,d,\varphi-\beta \psi)\geq0\}.
\end{align}
	
Combining the facts (\ref{inequ 2.12}), (\ref{inequ 2.13}) and Corollary \ref{coro 2.5}, we finish the proof.

\end{proof}
Now, we  are ready to prove the Theorem \ref{thm 1.1} and Theorem \ref{thm 1.2}. 
\begin{proof}[Proof of Theorem \ref{thm 1.1}]
	By Proposition \ref{prop 2.6},  the equation $\overline{{mdim}}_{M}(X,f,d,\varphi -\beta \psi)=0$ has unique root $\beta$. By Corollary \ref{coro 2.7}, we know the root $\beta$ is exactly equal to  $\overline{{mdim}}_{M,\psi}(X,f,d,\varphi)$.
\end{proof}

\begin{proof}[Proof of  Theorem \ref{thm 1.2}]
	It suffices to show  $$\overline{{mdim}}_{M,\psi}(X,f,d,\varphi)=\sup_{\mu\in M(X,f)}\limits \left\{\frac{\overline{rdim}(X,f,d,\mu)}{\int\psi d\mu}+\frac{\int \varphi d\mu}{\int\psi d\mu}\right\},$$
	and the remaining equality can be obtained similarly.
	
	Firstly, we show \emph{LHS} $\geq$ \emph {RHS}.
	Let $\beta >\overline{{mdim}}_{M,\psi}(X,f,d,\varphi)$. By Corollary \ref{coro 2.7}, we have
	\begin{align*}
		0&\geq \overline{mdim}_{M}(X,f,d,\varphi-\beta\psi)\\
		&=\sup_{\mu\in M(X,f)}\left\{\overline {r
			dim}(X,f,d,\mu)+\int \varphi d\mu -\beta  \int \psi d\mu\right\}~~\mbox{using Theorem A }\\
		&=\sup_{\mu\in M(X,f)}\left\{\int \psi d\mu \left(\frac{\overline {rdim}(X,f,d,\mu)}{\int \psi d\mu}+\frac{\int \varphi d\mu}{\int \psi d\mu}-\beta\right) \right\},	
	\end{align*}
    which implies that  $\frac{\overline {r
		dim}(X,f,d,\mu)}{\int \psi d\mu}+\frac{\int \varphi d\mu}{\int \psi d\mu}\leq\beta$ for all $\mu \in M(X,f)$. This shows the inequality \emph{LHS} $\geq$ \emph {RHS}.

Next, we prove the converse inequality   \emph{LHS} $\leq$ \emph {RHS} by using the same method. Let $\beta <\overline{{mdim}}_{M,\psi}(X,f,d,\varphi)$. By Corollary \ref{coro 2.7}, we have
\begin{align*}
	&\overline{mdim}_{M}(X,f,d,\varphi-\beta\psi)\\
	=&\sup_{\mu\in M(X,f)}\left\{\overline {rdim}(X,f,d,\mu)+\int \varphi d\mu -\beta  \int \psi d\mu\right\}~~\mbox{using  Theorem A}\\
	=&\sup_{\mu\in M(X,f)}\left\{\int \psi d\mu \left(\frac{\overline {rdim}(X,f,d,\mu)}{\int \psi d\mu}+\frac{\int \varphi d\mu}{\int \psi d\mu}-\beta\right) \right\}	\geq 0,
\end{align*}
which yields that $\frac{\overline {rdim}(X, f, d, \mu)}{\int \psi d\mu}+\frac{\int \varphi d\mu}{\int \psi d\mu}\geq\beta$ for some  $\mu \in M(X,f)$. This shows  the inequality \emph{LHS} $\leq$ \emph {RHS}.
This completes the proof.
	
\end{proof}

\section{The metric mean dimension  with  potential on subsets}
   The section 3 is divided into four parts.  In subsection 3.1, we recall some basic definitions of upper metric mean dimension with potential  on subsets and collect some standard facts.   The subsection 3.2 is devoted to  establishing the Bowen's equations for upper metric mean dimension with potential on subsets. The subsection 3.3  is designed to obtain  variational  principles for BS  metric mean dimension and  Packing BS  metric mean dimension, and  in subsection 3.4 we  focus on the  upper metric mean dimension of generic points of ergodic measures. 

\subsection{Several types of upper metric mean dimension with potential }

We  first  recall  the definitions of the upper metric mean dimension of arbitrary subset of $X$ defined by Carath\rm$\bar {e}$odory  structures using covering method introduced by  Wang \cite{w21} and Cheng et al.  \cite{cls21}. Besides,  we  also  apply  the Packing method used in fractal geometry  to define the Packing  upper metric mean dimension with potential on subsets. Furthermore, some basic properties related by these quantities  are derived. 
\begin{definition}
	Let $0<\epsilon<1$ and $\lambda \in \mathbb{R}$. For $Z\subset X$,  $\varphi \in C(X,\mathbb{R})$ and $d\in \mathscr{D}(X)$,  we define
	$$M(f,d,Z,\varphi,\lambda,N,\epsilon)=\inf\left\{\sum_{i\in I}\limits  e^{-n_i \lambda+\log \frac{1}{\epsilon}\cdot\sup_{y\in B_{n_i}(x_i,\epsilon)}S_{n_i} \varphi(y)}\right\},$$
	where the infimum  is taken over all  finite or countable covers $\{B_{n_i}(x_i,\epsilon)\}_{i\in I}$ of $Z$ with $n_i \geq N.$
	
	$$\overline m(f,d,Z,\varphi,\lambda,N,\epsilon)=\inf\left\{\sum_{i\in I}\limits  e^{-N\lambda+\log \frac{1}{\epsilon}\cdot\sup_{y\in B_N(x_i,\epsilon)}S_N\varphi(y)}\right\},$$
	where the infimum  is taken over all  finite or countable  covers $\{B_{n_i}(x_i,\epsilon)\}_{i\in I}$ of $Z$ with $n_i = N.$

	Let
	\begin{align*}
		M(f,d,Z,\varphi,\lambda,\epsilon)&=\lim\limits_{N\to \infty}M(f,d,Z,\varphi,\lambda,N,\epsilon),\\
		\overline{m}(f,d,Z,\varphi,\lambda,\epsilon)&=\limsup_{N\to \infty} \limits \overline {m}(f,d,Z,\varphi,\lambda,N,\epsilon).	
	\end{align*}
  It is readily  to check that  $M(f,d,Z,\varphi,\lambda,\epsilon),  \overline{m}(f,d,Z,\varphi,\lambda,\epsilon)$  have a critical value  of parameter $\lambda$  jumping from $\infty$ to $0$. We  respectively denote their critical values as
	\begin{align*}
		&\overline{{mdim}}_{M,Z,f}(\varphi,d,\epsilon):=\inf\{\lambda:M(f,d,Z,\varphi,\lambda,\epsilon)=0\}=\sup\{\lambda:M(f,d,Z,\varphi,\lambda,\epsilon)=\infty\},\\
		&\overline{{upmdim}}_{M,Z,f}(\varphi,d,\epsilon):=\inf\{\lambda:\overline{m}(f,d,Z,\varphi,\lambda,\epsilon)=0\}=\sup\{\lambda:\overline{m}(f,d,Z,\varphi,\lambda,\epsilon)=\infty\}.	
	\end{align*}
  Put
  \begin{align*}
	\overline{{mdim}}_{M,Z,f}(\varphi,d)&=\limsup_{\epsilon \to 0}\frac{\overline{{mdim}}_{M,Z,f}(\varphi,d,\epsilon)}{\log \frac{1}{\epsilon}},\\
	\overline{{upmdim}}_{M,Z,f}(\varphi,d)&=\limsup_{\epsilon \to 0}\frac{\overline{{upmdim}}_{M,Z,f}(\varphi,d,\epsilon)}{\log \frac{1}{\epsilon}}.
 \end{align*}

The quantities $\overline{{mdim}}_{M,Z,f}(\varphi,d),\overline{{upmdim}}_{M,Z,f}(\varphi,d)$ are called \emph{Bowen upper metric mean dimension with potential $\varphi$, u-upper metric  mean dimension with potential $\varphi$ on the set $Z$}, respectively.  Furthermore, we  sometimes omit  $d$  in these quantities when $d$ is clear. Specially, $\overline{{mdim}}_{M}(f,Z,d):=\overline{mdim}_{M,f,Z}(0,d)$  is called the \emph{Bowen upper metric mean dimension on $Z$}.
		
\end{definition}

\begin{remark}\label{ref 3.2}
	Let $Z\subset X$. Define 
	$$\overline{mdim}_{M}(Z,f,d,\varphi):=\limsup_{\epsilon \to 0}\limsup_{n \to \infty}\frac{1}{n\log{\frac{1}{\epsilon}}}\log\inf_{E_n}\left\{\sum_{x \in E_n}\limits e^ {\log \frac{1}{\epsilon
		}\cdot S_n\varphi(x)} \right\},$$ 
	where  the infimum  $E_n$ ranges over all $ (n,\epsilon)$-spanning sets of $Z$.
	
	By a standard method, one can check 
	\begin{align}
		\overline{mdim}_{M}(Z,f,d,\varphi)&=\limsup_{\epsilon \to 0}\limsup_{n \to \infty}\frac{1}{n\log{\frac{1}{\epsilon}}}\log\sup_{F_n}\left\{\sum_{x \in F_n}\limits e^ {\log \frac{1}{\epsilon
			}\cdot S_n\varphi(x)} \right\}\\
		&=\overline{{upmdim}}_{M,Z,f}(\varphi,d),
	\end{align}
	where  the supremum  $F_n$ ranges over all $ (n,\epsilon)$-separated sets of $Z$.
	
	Using the  fact \cite[Proposition 2.2]{cls21} that if $Z$ is a $f$-invariant compact subset, then $\overline{{mdim}}_{M,Z,f}(\varphi)=  \overline{{upmdim}}_{M,Z,f}(\varphi)$. Hence $$\overline{mdim}_{M}(X,f,d,\varphi)=\overline{mdim}_{M,X,f}(\varphi, d).$$
\end{remark} 

\begin{definition}
	Let $0<\epsilon<1$ and $\lambda \in \mathbb{R}$. For $Z\subset X$, $\varphi \in C(X,\mathbb{R})$ and $d\in \mathscr{D}(X)$,  we define
	$$P(f,d,Z,\varphi,\lambda,N,\epsilon)=\sup\left\{\sum_{i\in I}\limits  e^{-n_i \lambda+\log \frac{1}{\epsilon}\cdot\sup_{y\in \overline B_{n_i}(x_i,\epsilon)}S_{n_i} \varphi(y)}\right\},$$
	where the supremum is taken over all  finite or countable  pairwise disjoint  closed  families $\{\overline B_{n_i}(x_i,\epsilon)\}_{i\in I}$ of $Z$ with $n_i \geq N, x_i\in Z.$
	
	The quantity $P(f,d,Z,\varphi,\lambda,N,\epsilon)$ is non-increasing as $N$ increases, so we define
	$$P(f,d,Z,\varphi,\lambda,\epsilon)=\lim_{N \to \infty }\limits P(f,d,Z,\varphi,\lambda,N,\epsilon).$$
	Set
	$$\mathcal P(f,d,Z,\varphi,\lambda,\epsilon)=\inf\left\{\sum_{i=1}^{\infty}P(f,d,Z_i,\varphi,\lambda,\epsilon): \cup_{i\geq 1}Z_i \supseteq Z \right\}.$$
  It is readily  to check that the quantity $\mathcal P(f,d,Z,\varphi,\lambda,\epsilon)$  has a critical value  of parameter $\lambda$  jumping from  $\infty$ to $0$. We  define the critical value as
	\begin{align*}
			&\overline{{Pmdim}}_{M,Z,f}(\varphi,d,\epsilon):=\inf\{\lambda: \mathcal P(f,d,Z,\varphi,\lambda,\epsilon)=0\}=\sup\{\lambda: \mathcal P(f,d,Z,\varphi,\lambda,\epsilon)=\infty\}.
	\end{align*}

Let  $\overline{{Pmdim}}_{M,Z,f}(\varphi,d)=\limsup_{\epsilon \to 0}\limits\frac{\overline{{Pmdim}}_{M,Z,f}(\varphi,d,\epsilon)}{\log \frac{1}{\epsilon}}.$ We call  the quantities  $\overline{{Pmdim}}_{M,Z,f}(\varphi,d),\\
 \overline{{Pmdim}}_{M}(f,Z,d):=\overline{{Pmdim}}_{M,Z,f}(0,d)$  \emph{Packing upper metric mean dimension with potential  $\varphi$ on the set $Z$,  Packing upper metric mean dimension on the set $Z$}, respectively. We  sometimes omit the metric $d$ in above quantities  when $d$ is clear.
\end{definition}

 The following proposition presents some basic properties related by these  quantities. 
\begin{proposition}
	Let  $(X,f)$ be a TDS  with a metric $d\in \mathscr{D}(X)$ and $\varphi \in C(X,\mathbb{R})$. 
	\begin {enumerate}[(i)]
	\item  	If $Z_1\subset Z_2\subset X$, then $\overline{{mdim}}_{M,Z_1,f}(\varphi)\leq \overline{{mdim}}_{M,Z_2,f}(\varphi), \overline{{upmdim}}_{M,Z_1,f}(\varphi)\leq \overline{{upmdim}}_{M,Z_2,f}(\varphi), \\
	\overline{{Pmdim}}_{M,Z_1,f}(\varphi,d)\leq \overline{{Pmdim}}_{M,Z_2,f}(\varphi,d),$

	 \item  If   $Z$ is a finite union of some $Z_i$, that is, Z=$\cup_{i=1}^{N}Z_i$, then  $\overline{{mdim}}_{M,Z,f}(\varphi)=\\ \max_{1\leq i\leq N }\limits\overline{{mdim}}_{M,Z_i,f}(\varphi),
	 \overline{{upmdim}}_{M,Z,f}(\varphi)= \max_{1\leq i\leq N }\limits\overline{{upmdim}}_{M,Z_i,f}(\varphi),
	 \overline{{Pmdim}}_{M,Z,f}(\varphi)\\= \max_{1\leq i\leq N }\limits\overline{{Pmdim}}_{M,Z_i,f}(\varphi).$
	 
	\item For  any non-empty subset $Z\subset X$,
	$$\overline{{mdim}}_{M,Z,f}(\varphi)\leq  \overline{{Pmdim}}_{M,Z,f}(\varphi)\leq \overline{{upmdim}}_{M,Z,f}(\varphi).$$
	 Further, if $Z$ is compact and $f$-invariant, then 
	 $$\overline{{mdim}}_{M,Z,f}(\varphi)=\overline{{Pmdim}}_{M,Z,f}(\varphi)=\overline{{upmdim}}_{M,Z,f}(\varphi).$$
\end{enumerate}
\end{proposition}

\begin{proof}
	(i) and (ii) follow  directly from the Definitions 3.1 and 3.3.
	
	(iii)
	Let $0<\epsilon<1$ and $\gamma(4\epsilon)=\sup\{|\varphi(x)-\varphi(y)|: d(x,y)\leq 4\epsilon\}$, and let 
	$n\in \mathbb{N}$ and $A\subset X$. Let $R$ be the largest cardinality such that there exists a pairwise disjoint family $\{\overline  B_n(x_{i},\epsilon)\}_{i=1}^{R}$ with $x_i\in A$. Then 
	$$\cup_{i=1}^{R}  B_n(x_{i},3\epsilon)\supseteq A.$$
	Let $\lambda \in \mathbb{R}$, then
	\begin{align*}
	M(f,d,A,\varphi,\lambda,n,3\epsilon)&\leq \sum_{i=1}^{R} e^{-n\lambda+\log \frac{1}{3\epsilon}\cdot\sup_{y\in B_{n}(x_i,3\epsilon)} S_{n}\varphi(y)}\\
	&\leq \sum_{i=1}^{R} e^{-n\lambda+\log \frac{1}{3\epsilon}\cdot\sup_{y\in \overline B_{n}(x_i,\epsilon)} S_{n}\varphi(y)+\log \frac{1}{3\epsilon}\cdot n\gamma(4\epsilon)}\\
	&\leq \sum_{i=1}^{R} e^{-n\lambda+\log \frac{1}{\epsilon}\cdot \sup_{y\in \overline B_{n}(x_i,\epsilon)} S_{n}\varphi(y)-\log \frac{1}{3}\cdot n||\varphi||+\log \frac{1}{3\epsilon}\cdot n\gamma(4\epsilon)}\\
	&\leq P(f,d,A,\varphi,\lambda-\log 3||\varphi||-\log \frac{1}{3\epsilon}\cdot\gamma(4\epsilon),n,\epsilon).
	\end{align*}
	Hence  for any $\cup_{i\geq1}Z_i\supseteq Z$, we have 
	\begin{align*}
 M(f,d,Z,\varphi,\lambda,3\epsilon)&\leq\sum_{i\geq 1} M(f,d,Z_i,\varphi,\lambda,3\epsilon)\\
 &\leq\sum_{i\geq 1} P(f,d,Z_i,\varphi,\lambda-\log 3||\varphi||-\log \frac{1}{3\epsilon}\cdot \gamma(4\epsilon),\epsilon).
	\end{align*}
	This implies that 
	$$\overline{mdim}_{M,Z,f}(\varphi,d,3\epsilon)\leq \overline{{Pmdim}}_{M,Z,f}(\varphi, d,\epsilon)+\log 3||\varphi||+\gamma(4\epsilon)\log \frac{1}{3\epsilon}.$$
Therefore,  we  finally obtain  $\overline{{mdim}}_{M,Z,f}(\varphi)\leq \overline{{Pmdim}}_{M,Z,f}(\varphi)$.

We  continue to verify that $\overline{{Pmdim}}_{M,Z,f}(\varphi) \leq\overline{{upmdim}}_{M,Z,f}(\varphi)$. We may assume that $\overline{{Pmdim}}_{M,Z,f}(\varphi)>-\infty$, otherwise there  is nothing left  to prove. Let $-\infty<t<s<\overline{{Pmdim}}_{M,Z,f}(\varphi)$. Then we can choose a subsequence $0<\epsilon_k<1$ that convergences to 0 as $k \to \infty$ such that 
$$\overline{{Pmdim}}_{M,Z,f}(\varphi,d)=\lim_{k \to \infty}\frac{\overline{{Pmdim}}_{M,Z,f}(\varphi,d,\epsilon_k)}{\log \frac{1}{\epsilon_k}}>s.$$

Therefore,  there is $K_0 \in \mathbb{N}$ satisfying  for any $k>K_0$,
$$\overline{{Pmdim}}_{M,Z,f}(\varphi,d,\epsilon_k)>s\log \frac{1}{\epsilon_k}.$$
This means that  $P(f,d,Z,\varphi,s\log \frac{1}{\epsilon_k},\epsilon_k)\geq \mathcal{P} (f,d,Z,\varphi,s\log \frac{1}{\epsilon_k},\epsilon_k)=\infty$.

 Fix such a  $k>K_0$. For any $N \in \mathbb{N}$,  we can find a countable pairwise disjoint  family $\{\overline B_{n_i}(x_i,\epsilon_k)\}_{i \in I}$  with $x_i \in Z$ and $n_i\geq N$ such that 

$$\sum_{i\in I}  e^{-n_i \cdot s\log \frac{1}{\epsilon_k}+\log \frac{1}{\epsilon_k}\cdot\sup_{y\in \overline B_{n_i}(x_i,\epsilon_k)}S_{n_i} \varphi(y)}>1.$$

For any  $ l\geq N$, we set  $E_l=\{x_{n_i}: n_i=l, i\in I\}$. So
\begin{align*}
	&\sum_{l\geq N} \sum_{x\in E_l}  e^{-l \cdot s\log \frac{1}{\epsilon_k}+\log \frac{1}{\epsilon_k}(S_l\varphi(x)+l\gamma(\epsilon_k))}\\
	\geq &\sum_{l\geq N}\sum_{x\in E_l}  e^{-l \cdot s\log \frac{1}{\epsilon_k}+\log \frac{1}{\epsilon_k}\cdot\sup_{y\in \overline B_{l}(x,\epsilon_k)}S_{l} \varphi(y)}>1,
\end{align*}
where $\gamma(\epsilon):=\sup\{|\varphi(x)-\varphi(y)|:d(x,y)\leq\epsilon\}.$

There must exist an $l_N\geq N$ such that 
$$\sum_{x\in E_{l_N}} e^{-l_N (s-\gamma(\epsilon_k)) \log \frac{1}{\epsilon_k}+\log \frac{1}{\epsilon_k}S_{l_N}\varphi(x)}>(1-e^{(t-s)\log \frac{1}{\epsilon_k}})e^{(t-s)l_N\log \frac{1}{\epsilon_k}}.$$
Namely, we get $\sum_{x\in E_{l_N}}   (1/\epsilon_k)^{S_{l_N}\varphi(x)}>(1-e^{(t-s)\log \frac{1}{\epsilon_k}})(1/\epsilon_k)^{(t-\gamma(\epsilon_k))l_N},$
where $E_{l_N}$ is an $(l_N,\epsilon_k)$-separated set of $Z$. 
This gives us that 
$$\limsup_{N \to \infty }\frac{1}{N\log \frac{1}{\epsilon_k}}\log \sup_{E_N}\left\{\sum_{x\in E_N}(1/\epsilon_k)^{S_N\varphi(x)}\right\}\geq t-\gamma(\epsilon_k),$$
where the supremum ranges over all  $(N,\epsilon_k)$-separated sets of $Z$.
 
 Note that $\gamma(\epsilon_k) \to 0$  as $k \to \infty$, combining the fact mentioned in  remark \ref{ref 3.2}, we finally deduce that $\overline{{upmdim}}_{M,Z,f}(\varphi,d)\geq t$. Letting $t \to \overline{{Pmdim}}_{M,Z,f}(\varphi,d)$, we  get  the desired result.

The  last  statement follows   from (iii) and 
 the  fact  \cite[Proposition 2.2]{cls21} stating that if $Z$ is a $f$-invariant compact subset, then $\overline{{mdim}}_{M,Z,f}(\varphi)=  \overline{{upmdim}}_{M,Z,f}(\varphi)$.
	
\end{proof}

\subsection{Bowen's equation for upper metric mean dimension with potential on subsets}

 We begin   this subsection  with studying some basic properties  of the functions defined by the  Bowen upper metric mean dimension with potential  and Packing upper metric mean dimension with potential  on a subset of $X$. Then  we define BS  metric mean dimension and Packing BS metric mean dimension and show that  they are exactly the  unique root of  the corresponding Bowen's equations.

Given a non-empty subset $Z\subset X$ that does not need to be invariant or compact,  and  let   $\varphi \in C(X,\mathbb{R})$. Consider the following  functions 
\begin{align*}
	\phi(t)&=\overline{{mdim}}_{M,Z,f}(t\varphi,d), \\ 
	\Phi(t)&=\overline{{Pmdim}}_{M,Z,f}(t\varphi, d). 
	\end{align*}

\begin{proposition}
	Let $(X, f)$ be a TDS and   $Z\subset X$ be a  non-empty subset.  Suppose that $\varphi \in C(X,\mathbb{R})$ with $\varphi <0$. Then  for all $t \in \mathbb{R}$, one has  $\overline{{mdim}}_{M,Z,f}(t\varphi)>-\infty$,  and  $\overline{{mdim}}_{M,Z,f}(t\varphi)<\infty$ if and only if    $\overline{{mdim}}_{M}(f,Z)<\infty$.
\end{proposition}

\begin{proof}
	Set $m=\min_{x\in X}\varphi(x)$. Let  $0<\epsilon<1$ and $t\geq0$.  Then for each $N$,  we have 
	\begin{align*}
		M(f,d,Z,t\varphi,tm \log \frac{1}{\epsilon},N,\epsilon)&=\inf\left\{\sum_{i\in I}\limits  e^{-n_i tm \log \frac{1}{\epsilon}+t\log \frac{1}{\epsilon}\cdot\sup_{y\in B_{n_i}(x_i,\epsilon)}S_{n_i} \varphi(y)}\right\}\\
		&\geq \inf\left\{\sum_{i\in I}\limits  e^{-n_i tm \log\frac{1}{\epsilon}+n_i t m\log \frac{1}{\epsilon} }\right\}\geq1,
	\end{align*}
where the infimum  ranges over  all  finite or countable covers $\{B_{n_i}(x_i,\epsilon)\}_{i\in I}$ of $Z$ with $n_i \geq N$.

Therefore,  $\overline{{mdim}}_{M,Z,f}(t\varphi)\geq tm>-\infty.$ Now, fix a  $t_0>0$. Then for all $t<0$, we have $\overline{{mdim}}_{M,Z,f}(t\varphi)\geq \overline{{mdim}}_{M,Z,f}(t_0\varphi)>-\infty$ by the monotonicity of  $M(f,d,Z,\varphi,\lambda,N,\epsilon)$ with respect to  $\varphi$.	

For  the second statement,  
\begin{align*}
	\inf\left\{\sum_{i\in I}\limits  e^{-n_i \lambda-|t|n_i ||\varphi||\log \frac{1}{\epsilon}}\right\}
	&\leq
	\inf\left\{\sum_{i\in I}\limits  e^{-n_i \lambda+t\log \frac{1}{\epsilon}\cdot\sup_{y\in B_{n_i}(x_i,\epsilon)}S_{n_i} \varphi(y)}\right\}\\
	&\leq\inf\left\{\sum_{i\in I}\limits  e^{-n_i \lambda+|t|n_i ||\varphi||\log \frac{1}{\epsilon}}\right\},	
\end{align*}
where the infimum ranges over  all  finite or countable covers $\{B_{n_i}(x_i,\epsilon)\}_{i\in I}$ of $Z$ with $n_i \geq N$.

Note that $\overline{{mdim}}_{M}(f,Z)=\overline{{mdim}}_{M,Z,f}(0)$, this implies that 
$$ \overline{{mdim}}_{M}(f,Z)-|t| ||\varphi||\leq \overline{{mdim}}_{M,Z,f}(t\varphi)\leq  \overline{{mdim}}_{M}(f,Z)+|t| ||\varphi||.$$
Hence for all $t\in \mathbb{R}$, $\overline{{mdim}}_{M,Z,f}(t\varphi)<\infty$ if and only if    $\overline{{mdim}}_{M}(f,Z)<\infty$.

\end{proof}

\begin{proposition}\label{prop 3.6}
	Let $\varphi$ be a negative and continuous function on $X$. Suppose that\\
	$\overline{mdim}_{M}(f,X)<\infty$.
  Then the function $\phi(t)$ is strictly decreasing and Lipschitz, 
 the equation    $\phi(t)=0$  has unique (finite) root $s$ and  $-\frac{1}{m}  \overline{{mdim}}_{M}(f,Z) \leq s \leq -\frac{1}{M}\overline{{mdim}}_{M}(f,Z)$, where 	$m=\min_{x\in X}\varphi(x)$ and $M=\max_{x\in X}\varphi(x)$. 
\end{proposition}
\begin{proof}
	Let $t_1, t_2\in \mathbb{R}$ with $t_1>t_2$. Let $0<\epsilon<1$ and $N\in \mathbb{N}$.  Given   a  cover  $\{B_{n_i}(x_i,\epsilon)\}_{i\in I} $  of $Z$ with $n_i\geq N$, then  we have 
	\begin{align*}
		&\sum_{i\in I}\limits  e^{-n_i \lambda+t_1\log \frac{1}{\epsilon}\sup_{y\in B_{n_i}(x_i,\epsilon)}S_{n_i} \varphi(y)}\\
		\leq&\sum_{i\in I}\limits  e^{-n_i \lambda+t_2\log \frac{1}{\epsilon}\sup_{y\in B_{n_i}(x_i,\epsilon)}S_{n_i} \varphi(y)+(t_1-t_2)n_iM\log \frac{1}{\epsilon}}.	
	\end{align*}
	From this relation,  we deduce that 
	\begin{align}\label{inequ 3.3}
		\overline{{mdim}}_{M,Z,f}(t_1\varphi)\leq \overline{{mdim}}_{M,Z,f}(t_2\varphi)+(t_1-t_2)M,
	\end{align}
		which implies that $\phi(t)$ is strictly decreasing with respect to $t$ on $\mathbb{R}$.
	
	Similarly, 
	\begin{align}\label{inequ 3.4}
		\overline{{mdim}}_{M,Z,f}(t_2\varphi)+(t_1-t_2)m \leq \overline{{mdim}}_{M,Z,f}(t_1\varphi).
	\end{align}	
	Taking Lipschitz constant  $L:=-m$, we see that 
	
	$$|\overline{{mdim}}_{M,Z,f}(t_1\varphi)- \overline{{mdim}}_{M,Z,f}(t_2\varphi)| \leq L|t_1-t_2|.$$
	Letting $t_1=h>0, t_2=0$ in (\ref{inequ  3.3}),  then
	$$\overline{{mdim}}_{M,Z,f}(h\varphi)\leq \overline{{mdim}}_{M}(f,Z)-h(-M).$$
	
	Therefore, $\overline{{mdim}}_{M,Z,f}((\frac{1}{-M}\overline{{mdim}}_{M}(f,Z))\cdot\varphi)\leq0.$
	Again, letting $t_1=h>0, t_2=0$ in (\ref{inequ 3.4}), then
	$$\overline{{mdim}}_{M,Z,f}(h\varphi)\geq \overline{{mdim}}_{M}(f,Z)-h(-m).$$
	This gives us that $\overline{{mdim}}_{M,Z,f}((\frac{1}{-m}\overline{{mdim}}_{M}(f,Z))\cdot\varphi)\geq0.$
	
	Using the intermediate value theorem of continuous function, we know that the equation $\phi(t)=0$ has unique non-negative root $s$  and 
	$$-\frac{\overline{{mdim}}_{M}(f,Z)}{m}\leq s\leq -\frac{\overline{{mdim}}_{M}(f,Z)}{M}<\infty.$$

\end{proof}
By  slightly  modifying  the  method used in Proposition  3.6,  we have the following 
\begin{proposition}
	Let $\varphi$ be a negative and continuous function on $X$. Suppose that\\ $ \overline{Pmdim}_{M}(f,X) 
	<\infty$. Then the function $\Phi(t)$  is   strictly decreasing and Lipschitz. Moreover,  the equation  $\Phi(t)=0$   has unique root. 
\end{proposition}

Analogous to the setting of BS dimension \cite{bs00} and Packing BS dimension  \cite{wc12} on arbitrary subset defined by Carath\rm$\bar {e}$odory  structures,  we  define  two new notions  called  BS metric mean dimension and  Packing  BS metric mean dimension on  subsets.

\begin{definition}
	 For $0<\epsilon<1,N\in \mathbb{N}, \lambda \in \mathbb{R}$, $Z\subset X$ and $\varphi\in C(X,\mathbb{R})$ with $\varphi>0, d \in \mathscr{D}(X)$, we define
	$$R(f,d,\varphi,\lambda, Z,N,\epsilon)=\inf\left\{\sum_{i\in I}e^{-\lambda\sup_{y\in B_{n_i}(x_i,\epsilon)} S_{n_i}\varphi(y)}\right\},$$
where the infimum  is taken over all  finite or countable covers $\{B_{n_i}(x_i,\epsilon)\}_{i\in I}$ of $Z$ with $n_i \geq N.$

Since	$R(f,d,\varphi,\lambda, Z,N,\epsilon)$ is non-decreasing  as $N$ increases,  we define 
$$R(f,d,\varphi,\lambda, Z,\epsilon)=\lim_{N\to \infty}R(f,d,\varphi,\lambda, Z,N,\epsilon).$$
There is a critical value of the parameter $\lambda$ that jumps from $\infty$ to 0.  We define such  critical value  $R(X,f,d,\varphi, Z,\epsilon)$ as 
\begin{align*}
	R(f,d,\varphi, Z,\epsilon)&=\inf\{\lambda: R(f,d,\varphi,\lambda, Z,\epsilon)=0\},\\
	&=\sup\{\lambda:R(f,d,\varphi,\lambda, Z,\epsilon)=\infty\}.
\end{align*}
Let $$\overline{BSmdim}_{M,Z,f}(\varphi,d)=\limsup_{\epsilon \to 0}\limits \frac{R(f,d,\varphi, Z,\epsilon)}{\log\frac{1}{\epsilon}}.$$

 The quantity \emph{$\overline{BSmdim}_{M,Z,f}(\varphi,d)$ is said to be BS  metric  mean  dimension on the set $Z$ with respect to $\varphi$ (or simply BS metric mean dimension).} We sometimes  omit $d$ and write $\overline{BSmdim}_{M,Z,f}(\varphi)$  instead of $\overline{BSmdim}_{M,Z,f}(\varphi,d)$ when  $d$ is clear.

\end{definition}

\begin{definition}
         For $0<\epsilon<1,N\in \mathbb{N}, \lambda \in \mathbb{R}$, $Z\subset X$ and $\varphi\in C(X,\mathbb{R})$ with $\varphi>0, d \in \mathscr{D}(X)$, we define
		$$P_p(f,d,\varphi,Z,\lambda,N,\epsilon)=\sup\left\{\sum_{i\in I}\limits  e^{-\lambda\sup_{y\in B_{n_i}(x_i,\epsilon)}S_{n_i} \varphi(y)}\right\},$$
		where the supremum is taken over all  finite or countable  pairwise disjoint  closed  families $\{\overline B_{n_i}(x_i,\epsilon)\}_{i\in I}$ of $Z$ with $n_i \geq N, x_i\in Z$.
		
		The quantity $P_p(f,d,\varphi,Z,\lambda,N,\epsilon)$ is non-increasing as $N$ increases, so we define
		$$P_p(f,d,\varphi,Z,\lambda,\epsilon)=\lim_{N \to \infty }\limits P_p(f,d,\varphi,Z,\lambda,N,\epsilon).$$
		Define 
		$$\mathcal P_p(f,d,\varphi,Z,\lambda,\epsilon)=\inf\{\sum_{i=1}^{\infty}P_p(f,d,\varphi,Z_i,\lambda,\epsilon): \cup_{i\geq 1}Z_i \supseteq Z \}.$$
	The quantity $\mathcal P_p(f,d,\varphi,Z,\lambda,\epsilon)$  has a critical value  of parameter $\lambda$  jumping from $\infty$ to $0$. We  define such  critical value as
		\begin{align*}
			\overline{{BSPmdim}}_{M,Z,f}(\varphi,d,\epsilon):&=\inf\{\lambda: \mathcal P_p(f,d,\varphi,Z,\lambda,\epsilon)=0\}\\
			&=\sup\{\lambda:\mathcal P_p(f,d,\varphi,Z,\lambda,\epsilon)=+\infty\}.
		\end{align*}
		Let 
		\begin{align*}
			\overline{{BSPmdim}}_{M,Z,f}(\varphi,d)&=\limsup_{\epsilon \to 0}\frac{\overline{{BSPmdim}}_{M,Z,f}(\varphi,d,\epsilon)}{\log \frac{1}{\epsilon}}.
		\end{align*}
	
	 \emph{$\overline{{BSPmdim}}_{M,Z,f}(\varphi,d)$ is called Packing   BS metric mean dimension  on the set $Z$ with respect to $\varphi$ (or simply Packing BS metric mean dimension)}, and we sometimes  omit $d$ and write $\overline{BSPmdim}_{M,Z,f}(\varphi)$  instead of $\overline{BSPmdim}_{M,Z,f}(\varphi,d)$ when  $d$ is clear.	
	
\end{definition}
\begin{remark}
	\begin {enumerate}[(i)]
	\item For any $Z\subset X$, $0\leq \overline{{BSmdim}}_{M,Z,f}(\varphi)\leq \overline{{BSPmdim}}_{M,Z,f}(\varphi).$
	\item 
	$\overline{BSmdim}_{M,Z,f}(1)=\overline{mdim}_{M}(f,Z),~
	\overline{BSPmdim}_{M,Z,f}(1)=\overline{Pmdim}_{M}(f,Z).$
	
\end{enumerate}
	
\end{remark}

 We now are ready to verify the Theorem 1.3.
\begin{proof}[Proof of   Theorem \ref{thm 1.3}]
	 Let $0<\epsilon <1$. Note that  for each $N$, $$M(f,d,Z,- \frac{\lambda\varphi}{\log\frac{1}{\epsilon}},0,N,\epsilon)=R(f,d,\varphi, Z,\lambda,N,\epsilon).$$
	Let $s>\overline{BSmdim}_{M,Z,f}(\varphi)$. Then  
	$R(f,d,\varphi, Z,s\log\frac{1}{\epsilon},\epsilon)<1$ for sufficiently small $\epsilon>0$. Hence  $M(f,d,Z,- s\varphi, 0,\epsilon)<1$, which implies that $\overline{{mdim}}_{M,Z,f}(-s\varphi)
	\leq0$. Using  the continuity  of $\phi$ obtained  in Proposition \ref{prop 3.6},   we obtain  $$\overline{{mdim}}_{M,Z,f}(-\overline{BSmdim}_{M,Z,f}(\varphi)\cdot\varphi)\leq0$$
	after letting $s\to\overline{BSmdim}_{M,Z,f}(\varphi)$.
	
	Let $s<\overline{BSmdim}_{M,Z,f}(\varphi)$. There exists a subsequence $0<\epsilon_k<1$   that convergences to 0  as $k \to \infty$ such that 
	$$\overline{{BSmdim}}_{M,Z,f}(\varphi,d)=\lim_{k \to \infty }\frac{\overline{{BSmdim}}_{M,Z,f}(\varphi,d,\epsilon_k)}{\log \frac{1}{\epsilon_k}}.$$ 
	It follows that  $$R(f,d,\varphi, Z,s\log\frac{1}{\epsilon_k},\epsilon_k)>1$$
	for  all sufficiently  large $k$. 
	This shows $M(f,d,Z,- s\varphi, 0,\epsilon_k)>1$. Similarly, we can deduce  that 
	
	$$\overline{{mdim}}_{M,Z,f}(-\overline{BSmdim}_{M,Z,f}(\varphi)\cdot\varphi)\geq0.$$
	Hence, by Proposition \ref{prop 3.6},    $\overline{BSmdim}_{M,Z,f}(\varphi)$ is the unique root of  the equation $\overline{{mdim}}_{M,Z,f}(-t\varphi)=0$.
	
	Using the relation  $\mathcal{P}(f,d,Z,-\frac{\lambda}{\log \frac{1}{\epsilon}}\varphi,0,\epsilon)=\mathcal{P}_p(f,d,\varphi,Z,\lambda,\epsilon)$, one can similarly deduce that  $\overline{BSPmdim}_{M,Z,f}(\varphi)$ is the unique root of  the equation 
	$\overline{{Pmdim}}_{M,Z,f}(-t\varphi)=0$.

\end{proof}

The  following corollary shows that  the   BS metric mean dimension  is a special case of $\psi$-induced upper metric mean dimension with potential  $0$.

\begin{corollary}
	Let $(X,f)$ be a TDS with a metric $d\in \mathscr{D}(X)$ and $\psi\in C(X,\mathbb{R})$ with $\psi>0$.  Then  $$\overline{{mdim}}_{M,\psi}(X,f,d,0)=\overline{BSmdim}_{M,X,f}(\psi).$$
\end{corollary}

\begin{proof}
	If $\overline{{mdim}}_M(f,X,d)=\infty$, by Remark \ref{ref 3.2}, then $$\overline{{mdim}}_M(f,X,d)=\overline{{mdim}}_{M,X,f}(0,d)=\overline{{mdim}}_{M,X,f}(0\cdot (-\psi),d)=\overline{{mdim}}_{M}(X,f,d,0\cdot(-\psi),d)=\infty.$$
Taking $\varphi=0$ in Proposition \ref{prop 2.6},  we  get
$$\overline{{mdim}}_{M}(X,f,d,-\beta \psi,d)=\infty$$
for all $\beta \in \mathbb{R}$.  By  Corollary \ref{coro 2.7}, we have $$\overline{{mdim}}_{M,\psi}(X,f,d,0)=\inf\{\beta\in \mathbb{R}:\overline{mdim}_{M}(X,f,d,-\beta\psi)\leq0\}=\inf\emptyset =\infty.$$

Set $M:=\max_{x\in X}\psi (x)>0$ and $\lambda \geq 0$.  For each $0<\epsilon<1$ and $N\in \mathbb{N}$, 
\begin{align*}
	R(f,d,\psi,\lambda,X,N,\epsilon)&=\inf\left\{\sum_{i\in I}e^{-\lambda\sup_{y \in B_{n_i}(x_i,\epsilon)} S_{n_i}\psi(y)}\right\}\\\
	&\geq \inf \left\{\sum_{i\in I} e^{-\lambda M n_i }\right\}=M(f,d,X,0,M\lambda,N,\epsilon),
\end{align*}
where the infimum  is taken over all  finite or countable covers $\{B_{n_i}(x_i,\epsilon)\}_{i\in I}$ of $X$ with $n_i \geq N.$

From this relation, we finally get  that 
$$\infty=\frac{\overline{{mdim}}_{M}(f,X,d)}{M}\leq \overline{{BSmdim}}_{M,X,f}(\psi,d).$$
 Therefore, 
 $$\overline{{mdim}}_{M,\psi}(X,f,d,0)=\overline{BSmdim}_{M,X,f}(\psi,d)=\infty.$$
 
For the case $\overline{{mdim}}_M(f,X,d)<\infty$, by remark \ref{ref 3.2}, we have 
$\overline{{mdim}}_M(f,X,d)=\overline{{mdim}}_M(X,f,d)\\
<\infty$.  By Theorem \ref{thm 1.1},  we have  $$\overline{{mdim}}_{M,X,f}(-\overline{{mdim}}_{M,\psi}(X,f,d,0)\cdot\psi,d )=\overline{{mdim}}_{M}(X,f,d,-\overline{{mdim}}_{M,\psi}(X,f,d,0)\cdot\psi,d )=0.$$ 
Combing with Theorem \ref{thm 1.3},  we obtain 
 $$\overline{{mdim}}_{M,\psi}(X,f,d,0)=\overline{BSmdim}_{M,X,f}(\psi,d)$$
 by  the uniqueness of the root of  the equation.
\end{proof}

	As a direct consequence of Theorem \ref{thm 1.2} and Corollary 3.11, we  have established  a variational principle for BS  metric mean dimension as follows. 

\begin{corollary}
Let $(X,f)$ be a  TDS  admitting marker property and $\psi \in C(X,\mathbb{R})$ with $\psi>0$. Then  for all $d\in \mathscr{D}^{'}(X)$, one has
	\begin{align*}
		\overline{BSmdim}_{M,X,f}(\psi,d)&=\sup_{\mu\in M(X,f)}\limits \left\{\frac{\underline{rdim}(X,f,d,\mu)}{\int\psi d\mu}\right\}\\
		&=\sup_{\mu\in M(X,f)}\limits \left\{\frac{\overline{rdim}(X,f,d,\mu)}{\int\psi d\mu}\right\}.
	\end{align*}
\end{corollary}
\subsection{Variational principles for BS and    Packing  BS  metric mean dimension on subsets}
  In Corollary 3.12, we  have established a variational principle for BS  metric mean dimension on the whole phase space in terms of rate distortion dimensions over invariant measures. In this subsection, we proceed to establish the variational principles for BS  metric mean dimension and  Packing  BS metric  mean dimension on subsets.  The following abundant critical ingredients are due to \cite{wc12}.
\begin{definition}\label{def 3.13}\cite[Definition 3.8]{wc12}
	Let  $\mu \in M(X)$,  $\varphi \in C(X,\mathbb{R})$ with $\varphi >0$, we define 
	
	\begin{align*}
		\underline{h}_{\varphi,\mu}(f,\epsilon)&=\int \liminf_{n \to \infty}-\frac{\log \mu(B_n(x,\epsilon))}{S_n\varphi(x) }d\mu,\\
		\overline{h}_{\varphi,\mu}(f,\epsilon)&=\int \limsup_{n \to \infty}-\frac{\log \mu(B_n(x,\epsilon))}{S_n\varphi(x) }d\mu.
	\end{align*} 

Let $\underline{h}_{\varphi,\mu}(f)=\lim_{\epsilon \to 0} \limits \underline{h}_{\varphi,\mu}(f,\epsilon)$,  $\overline{h}_{\varphi,\mu}(f)=\lim_{\epsilon \to 0} \limits \overline{h}_{\varphi,\mu}(f,\epsilon)$. We call the quantities $\underline{h}_{\varphi,\mu}(f),  \overline{h}_{\varphi,\mu}(f)$ \emph{the measure-theoretical  lower and upper BS  entropies  of $\mu$}, respectively.
\end{definition}

\begin{remark}
	If $\mu \in  E (X,f)$, by Birkhoff  ergodic theorem and Brin-Katok formula, then  $\underline{h}_{\varphi,\mu}(f)=\overline{h}_{\varphi,\mu}(f)=\frac{h_{\mu}(f)}{\int \varphi d\mu}$. When $\varphi =1$, the  measure-theoretical  lower and upper BS  entropies  of $\mu$  is reduced to the classical Brin-Katok  entropy formula \cite{bk83}. 
	
\end{remark}

\begin{lemma}\label{lem 3.15}\rm {\cite[Theorem 2.1]{m95}}
	Let $(X,d)$ be a compact metric space. Suppose that $\mathcal{B}=\{B(x_i, r_i)\}_{i\in I}$ is a family of open (or closed) balls in $X$. Then  there exists a finite or countable subfamily $\mathcal{B}^{'}=\{B(x_i,r_i)\}_{i\in I^{'}}$ of  pairwise disjoint  balls in $\mathcal{B}$ such that 
	
	$$\cup_{B\in \mathcal{B}}B \subseteq \cup_{i\in I^{'}}B(x_i,5r_i).$$
\end{lemma}

\begin{definition}
	Let $\varphi \in C(X,\mathbb{R})$ with $\varphi >0$ and $\psi$ be a non-negative bound function on $X$, and let $\lambda \in \mathbb{R}$ and $N\in \mathbb{N}, \epsilon >0$. Define 
	$$W(f,d,\varphi, \psi,\lambda,N,\epsilon)=\inf\{\sum_{i\in I}c_i e^{-\lambda \sup_{y\in B_{n_i}(x_i,\epsilon)}S_{n_i}\varphi (x)}\},$$
	where the infimum ranges over all finite or countable families $\{(B_{n_i}(x_i,\epsilon),c_i)\}_{i \in I}$ satisfying  $0<c_i<\infty$, $x_i \in X$, $n_i \geq N$, and 
	$$\sum_{i\in I}\limits c_i \chi_{B_{n_i}(x_i,\epsilon)}\geq \psi,$$
	where $\chi_{E}$ denotes the characteristic function of $E$.
	
	For $Z\subset X$, set $W(f,d,\varphi,Z,\lambda,N,\epsilon):=W(f,d,\varphi,\chi_Z,\lambda,N,\epsilon)$. Since the quantity \\$W(f,d,\varphi,Z,\lambda,N,\epsilon)$ is non-decreasing as $N$ increases,  so we define 

$$W(f,d,\varphi,Z,\lambda,\epsilon)=\lim_{N \to \infty } W(f,d,\varphi,Z,\lambda,N,\epsilon).$$
There is a critical value of $\lambda$ so that  $W(f,d,\varphi,Z,\lambda,\epsilon)$ jumps from $\infty$ to $0$. We define such  critical value as 
\begin{align*}
	\overline {Wmdim}_{M,Z,f}(\varphi,d,\epsilon):&=\inf\{\lambda : W(f,d,\varphi,Z,\lambda,\epsilon)=0\},\\
	&=\sup\{\lambda : W(f,d,\varphi,Z,\lambda,\epsilon)=\infty\}.
\end{align*}

Let  $	\overline {Wmdim}_{M,Z,f}(\varphi,d)=\limsup_{\epsilon \to 0} \frac{	\overline {Wmdim}_{M,Z,f}(\varphi,d,\epsilon)}{\log \frac{1}{\epsilon}}$,   and we call  the quantity  $\overline {Wmdim}_{M,Z,f}(\varphi,d)$ \emph{the weighted BS  metric mean dimension on the set $Z$ with respect  to $\varphi$}.

\end{definition}
Wang and  Chen \cite[Lemma 5.1]{wc12} proved the following proposition.
\begin{proposition}
		Let $(X,f)$ be a TDS with a metric $d\in \mathscr{D}(X)$ and $\varphi \in C(X,\mathbb{R})$ with $\varphi>0$, and let  $0<\epsilon <1$ and $Z\subset X$. Then 
	
	$$R(f,d,\varphi,Z,\lambda+\delta, N,6\epsilon)\leq W(f,d,\varphi,Z,\lambda,N,\epsilon)\leq R(f,d,\varphi,Z,\lambda,N,\epsilon)$$
	 holds for  all $\lambda >0$, $\delta >0.$
	 Consequently,  $\overline{BSmdim}_{M,Z,f}(\varphi,d)=\overline {Wmdim}_{M,Z,f}(\varphi,d).$
\end{proposition}

\begin{lemma}[\bf BS  Frostman's lemma]\rm{\cite[Lemma 6.1]{wc12}}
	Let $K$ be a non-empty compact subset of $X$ and $ \lambda \geq0, \epsilon >0, N\in \mathbb{N}$, $\varphi\in C(X,\mathbb{R})$ with $\varphi >0$. Suppose that $c:=W(f,d,\varphi, K,\lambda,N,\epsilon)>0$. Then there exists a Borel probability measure $\mu \in M(X)$ such that $\mu (K)=1$ and 
	
	$$\mu(B_n(x,\epsilon))\leq \frac{1}{c}e^{-\lambda S_n \varphi (x)}$$
	holds for all $x\in K, n\geq N$.
\end{lemma}

The following proposition   can be  proved  by following the line of  the  first part of the  proof given in \cite[Theorem 7.2]{wc12}.
\begin{proposition}
	Let $(X,f)$ be a TDS with a metric $d\in \mathscr{D}(X)$ and $\varphi \in C(X,\mathbb{R})$ with $\varphi>0$. Let  $K$ be a non-empty compact subset of $X$. Then 	for any  $0<\epsilon<1$ and  $\mu \in M(X)$ with  $\mu (K)=1$, we have  $(1-\frac{\gamma(2\epsilon)}{m})\underline{h}_{\varphi,\mu}(f,\epsilon) \leq R(f,d,\varphi, K,\frac{\epsilon}{2}),$ where $m=\min_{x\in X} \limits \varphi(x)>0$.	
\end{proposition}

Next, we give the proof of Theorem 1.4.
\begin{proof}[Proof of Theorem \ref{thm 1.4}]
	
	Firstly, we show $$\overline{{BSmdim}}_{M,K,f}(\varphi,d )=\limsup_{\epsilon \to 0}\frac{\sup \left\{\underline{h}_{\varphi, \mu}(f,\epsilon), \mu \in M(X), \mu (K)=1\right\}}{\log \frac{1}{\epsilon}}.$$ 
	
	It is clear that $LHS\geq RHS$   follows from the Proposition 3.19. On the other hand, we assume that $\overline{{BSmdim}}_{M,K,f}(\varphi,d )>0$. By Proposition 3.17,  we know that $\overline{BSmdim}_{M,K,f}(\varphi,d)=\overline {Wmdim}_{M,K,f}(\varphi,d)$. Let $0<\lambda < \overline {Wmdim}_{M,K,f}(\varphi,d)$. Then we can find a sequence $0<\epsilon_k<1$  that convergences to 0 as $k \to \infty$ so that 
	$$\overline{Wmdim}_{M,K,f}(\varphi,d)=\lim_{k\to \infty} \frac{\overline{Wmdim}_{M,K,f}(\varphi,d,\epsilon_k)}{\log \frac{1}{\epsilon_k}}> \lambda.$$
	Hence, for all  sufficiently large $k$, there is  $N_0\in \mathbb{N}$ such that
	$c:=W(f,d,\varphi,\lambda\log\frac{1}{\epsilon_k},Z,N_0,\epsilon_k)>0$.  By virtue of  Lemma 3.18, there exists a  Borel probability measure $\mu \in M(X)$ such that $\mu (K)=1$ and 
	
	$$\mu(B_n(x,\epsilon_k))\leq \frac{1}{c}e^{-\lambda \log \frac{1}{\epsilon_k}\cdot S_n \varphi (x)}$$
	holds for all $x\in X, n\geq N_0$. 
	
	This gives us that 
	$$\frac{\sup \left\{\underline{h}_{\varphi, \mu}(f,\epsilon_k), \mu \in M(X), \mu (K)=1\right\}}{\log \frac{1}{\epsilon_k}}\geq \frac{\underline{h}_{\varphi, \mu}(f,\epsilon_k)}{\log \frac{1}{\epsilon_k}} \geq \lambda. $$
	for all sufficiently  large $k$, which  implies that $ LHS \leq RHS $.
	
	Next, we  prove that 
	
	$$\overline{{BSPmdim}}_{M,K,f}(\varphi,d )=\limsup_{\epsilon \to 0}\frac{\sup \left\{\overline{h}_{\varphi, \mu}(f,\epsilon), \mu \in M(X), \mu (K)=1\right\}}{\log \frac{1}{\epsilon}}.$$
	Fix a sufficiently small $\epsilon$ with $0<\epsilon<1$.  We may assume that $\overline{{BSPmdim}}_{M,K,f}(\varphi,d, \epsilon )>0$. Let $0<s< \overline{{BSPmdim}}_{M,K,f}(\varphi,d, \epsilon )$. By \cite[Theorem 3.12, Part 2]{wc12},  there  is a $\mu \in M(X)$  with $\mu (K)=1$ such that for any $x\in K$, there exists a subsequence $n_i:=n_i(x)$ so that 
	$$\mu (B_{n_i}(x,\epsilon))\leq C\cdot e^{-s\cdot S_{n_i}\varphi(x) },$$ 
	where $C$ is a constant that does not depend on  the points of  $K$.
	
	It follows that $\overline h _{\varphi,\mu}(f,\epsilon)\geq s$,  and we obtain that 
	
	$$\overline h _{\varphi,\mu}(f,\epsilon)\geq \overline{{BSPmdim}}_{M,K,f}(\varphi,d, \epsilon ),$$ after letting $s \to \overline{{BSPmdim}}_{M,K,f}(\varphi,d, \epsilon )$, 
	which yields that $RHS\geq LHS$.

Let $ \mu \in M(X)$ with $\mu(K)=1$.  We assume that $\overline{h}_{\varphi,\mu}(f,2\epsilon)>0$. Let $0<s<\overline{h}_{\varphi,\mu}(f,2\epsilon)$. We can choose $\delta >0$  and a Borel set $A\subset K$ with $\mu(A)>0$ such that 

$$\limsup_{n \to \infty }-\frac{\log \mu(B_n(x,2\epsilon))}{S_n\varphi(x)}>s+\delta$$
for all $x\in A$.	

Next,  we show   $\mathcal{P}_p(f,d,K,\varphi,s(1-\frac{\gamma(\epsilon)}{m}),\frac{\epsilon}{5})=\infty$, where $m=\min_{x\in X}\varphi(x)>0$ and $\gamma(\epsilon)=\sup\{|\varphi(x)-\varphi(y)|: d(x,y)\leq\epsilon\}.$
To this end, it suffices to show for any $E\subset A$ with $\mu(E)>0$,  we have 
$P_p(f,d,\varphi,E, s(1-\frac{\gamma(\epsilon)}{m}),\frac{\epsilon}{5})=\infty $. Fix such a set $E$, define 

$$E_n:=\left\{x\in E: \mu(B_n(x,2\epsilon))<e^{-(s+\delta)S_n\varphi(x)}\right\}.$$

Then we have $E=\cup_{n\geq N}E_n$ for any $N\in\mathbb{N}$.  Fix such a $N$, by $\mu(E)=\mu(\cup_{n\geq N}E_n)$,  then there is a $n\geq N$ so that 
$$\mu(E_n)\geq \frac{1}{n(n+1)}\mu(E).$$
Fix such $n$, consider a  family of closed cover $\{\overline B_n(x,\frac{\epsilon}{5}):x\in E_n\}$ of $E_n$.  By Lemma \ref{lem 3.15} (replacing $d$ with the Bowen metric $d_n$), then  there exists  a finite  pairwise disjoint subfamily $\{\overline B_n(x_i,\frac{\epsilon}{5}):x_i\in E_n\}_{i\in I}$, where $I $ is a finite  index set, such that 
$$\cup_{i\in I}\overline B_n(x_i,\epsilon)\supseteq \cup_{x\in E_n}\overline B_n(x,\frac{\epsilon}{5})\supseteq E_n.$$

For each $i \in I$, we have
\begin{align*}
	\sup_{y\in \overline  B_n(x_i,\frac{\epsilon}{5})} S_n\varphi(y)&\leq S_n\varphi(x_i)+n\gamma(\epsilon)\\
	&\leq S_n\varphi(x_i)+\frac{ \sup_{y\in \overline B_n(x_i,\frac{\epsilon}{5})} \limits S_n\varphi(y)}{m} \gamma(\epsilon).
\end{align*}	
Hence, 
\begin{align*}
	P_p(f,d,\varphi,E,s(1-\frac{\gamma(\epsilon)}{m}),N,\frac{\epsilon}{5})&\geq 	P_p(f,d,\varphi,E_n,s(1-\frac{\gamma(\epsilon)}{m}),N,\frac{\epsilon}{5})\\
	&\geq \sum_{i\in I} e^{-s(1-\frac{\gamma(\epsilon)}{m})\sup_{y\in \overline B_n(x_i,\frac{\epsilon}{5})} S_n\varphi(y)}\\
	&\geq \sum_{i\in I} e^{-s S_n\varphi(x_i)}\\
	&= \sum_{i\in I} e^{-(s+\delta) S_n\varphi(x_i)}e^{\delta S_n\varphi(x_i)}\\
	&\geq e^{nm\delta} \sum_{i\in I} \mu(\overline B_n(x_i,\epsilon))\\
	&\geq e^{nm\delta} \mu (E_n)\\
	&\geq e^{nm\delta}\frac{\mu(E)}{n(n+1)}.
\end{align*}
Letting $N \to \infty$, we obtain that $P_p(f,d,E,s(1-\frac{\gamma(\epsilon)}{m}),\frac{\epsilon}{5})=\infty.$
This gives us that $$\overline{{BSPmdim}}_{M,K,f}(\varphi,d,\frac{\epsilon}{5})\geq s(1-\frac{\gamma(\epsilon)}{m}).$$
Letting $s \to \overline{h}_{\varphi,\mu}(f,2\epsilon)$, we have  $\overline{h}_{\varphi,\mu}(f,2\epsilon)(1-\frac{\gamma(\epsilon)}{m})\leq \overline{{BSPmdim}}_{M,K,f}(\varphi,d,\frac{\epsilon}{5})$ for all $\mu \in M(X)$ with $\mu(K)=1$. This implies that
$$(1-\frac{\gamma(\epsilon)}{m})\sup\{\overline{h}_{\varphi,\mu}(f,2\epsilon):\mu \in M(X), \mu(K)=1\}\leq\overline{{BSPmdim}}_{M,K,f}(\varphi,d,\frac{\epsilon}{5}),$$
which yields that $ LHS\geq RHS $.	
\end{proof}

\subsection{Bowen upper metric  mean dimension of  the set of generic points} \label{sub 3.4}
We first  collect several types of measure-theoretical entropies defined by invariant measures (or ergodic measures) as  candidates to characterize  the   Bowen upper metric mean dimension of  the sets of generic points of ergodic measures.

Let $(X, f)$ be a TDS with  a metric $d\in \mathscr{D}(X)$. Given $\mu \in M(X,f)$, by $h_{\mu}(f)$ we denote the   measure-theoretical entropy  of $\mu$.
\begin{enumerate}[(i)]
\item   Measure-theoretical entropy    given from the  viewpoint  of  the local perspective. 
Put 
\begin{align*}
	\underline h_{\mu}^{BK}(f,d, \epsilon)&=\int \liminf_{n \to \infty}-\frac{\log\mu(B_n(x,\epsilon))}{n}d\mu,\\
	\overline h_{\mu}^{BK}(f,d, \epsilon)&=\int \limsup_{n \to \infty}-\frac{\log\mu(B_n(x,\epsilon))}{n}d\mu. 
\end{align*}

Brin and Katok \cite{bk83}  showed that $h_{\mu}(f)=\lim_{\epsilon \to 0}\limits\underline h_{\mu}^{BK}(f,d,\epsilon)=\lim_{\epsilon \to 0}\limits\overline h_{\mu}^{BK}(f,d,\epsilon)$ for all $ \mu \in M(X,f)$.
If $\mu$ is an ergodic measure, they also showed that  for each   fixed   $\epsilon>0$,
	 $$\liminf_{n \to \infty}-\frac{\log\mu(B_n(x,\epsilon))}{n},  \limsup_{n \to \infty}-\frac{\log\mu(B_n(x,\epsilon))}{n} $$ are  both constants   for  $\mu $-a.e $x \in X$. In this  case,   we  still denote 
	\begin{align*}
		\underline h_{\mu}^{BK}(f,d,\epsilon)&=\liminf_{n \to \infty}-\frac{\log\mu(B_n(x,\epsilon))}{n},\\
		\overline h_{\mu}^{BK}(f,d,\epsilon)&=\limsup_{n \to \infty}-\frac{\log\mu(B_n(x,\epsilon))}{n}. 
	\end{align*}

\item   Measure-theoretical entropy  defined by  separated set and spanning set.

Put
$$PS(f,d,\mu,\epsilon)=\inf_{F\ni \mu}\limsup_{n \to \infty}\frac{1}{n}\log s_n(f,d,\epsilon,X_{n,F}),$$
where  the infimum runs over all neighborhoods of $\mu$ in $M(X)$ and  $X_{n,F}=\{x\in X: \frac{1}{n}\sum_{j=1}^n \delta_{f^{j}(x)} \in F\}$.

If $\mu \in E(X,f)$, Pfister and Sullivan  \cite{ps07} proved that $h_{\mu}(f)=\lim_{\epsilon \to 0}\limits PS(f,d,\mu,\epsilon)$.

Let $\delta \in (0,1)$. Put 
\begin{align*}
	\overline h^K_{\mu}(f,d,\epsilon,\delta)=\limsup_{n \to \infty}\frac{1}{n}\log r_n(\mu; d,\epsilon,\delta),\\
	\overline h^K_{\mu}(f,d,\epsilon)=\lim_{\delta \to 0}\limsup_{n \to \infty}\frac{1}{n}\log r_n(\mu; d,\epsilon,\delta),
\end{align*}
where  $r_n(\mu; d,\epsilon,\delta)=\min\{\#F: \mu (\cup_{x\in F}B_n(x,\epsilon))>1-\delta, F\subset X\}.$

For $\mu \in E(X,f)$,  Katok  \cite{k80} showed that $h_{\mu}(f)=\lim_{\epsilon \to 0}\limits \overline h^K_{\mu}(f,d,\epsilon,\delta)$ for any $\delta\in (0,1)$.

\item  The last candidate comes from information theory. Recall that  \emph{upper and lower rate distortion dimensions} are respectively given by 

$$\overline{rdim}(X,f,d,\mu)=\limsup_{\epsilon \to 0}\frac{R(d,\mu,\epsilon)}{\log \frac{1}{\epsilon}},$$
$$\underline{rdim}(X,f,d,\mu)=\liminf_{\epsilon \to 0}\frac{R(d,\mu,\epsilon)}{\log \frac{1}{\epsilon}},$$
where  $R(d,\mu,\epsilon)$  denotes
the rate distortion function. 

Replacing $R(d,\mu,\epsilon)$  by $R_{L^{\infty}}(d,\mu,\epsilon)$, one can  similarly  define   upper $L^{\infty}$-rate distortion dimension  $\overline{rdim}_{L^{\infty}}(X,f,d,\mu)$ and lower $L^{\infty}$-rate distortion dimension  $\underline{rdim}_{L^{\infty}}(X,f,d,\mu)$,
where   $R_{L^{\infty}}(d,\mu,\epsilon)$ denotes
$L^{\infty}$-rate distortion function.   Due to the forthcoming  proof does not refer to the definitions  of $R(d,\mu,\epsilon)$   and $R_{L^{\infty}}(d,\mu,\epsilon)$ ,  we  omit  their precise definitions   and  refer readers to \cite{ct06, lt18, lt19} for   more details. 

\end{enumerate}

Inspired by the method used in \cite[Theorem 1.2]{zc18}, we proceed to prove Theorem 1.5, (i). 

\begin{proof}[Proof of  Theorem \ref{thm 1.5}, (i)]
Let $\mu \in M(X,f)$ with $\mu(Y)=1$.  There exists an increasing sequence $Y_n$ of compact subsets  of $Y$ satisfying $\mu (Y_n)>1-\frac{1}{n}$ for all $n\in \mathbb{N}$. 	

Therefore,
$$\overline{{mdim}}_{M,f,Y}(0,d,\epsilon)\geq\overline{{mdim}}_{M,f, \cup_{n\geq 1} Y_n}(0,d,\epsilon)=\lim_{n \to \infty}\overline{{mdim}}_{M,f,Y_n}(0,d,\epsilon). $$

Put $\mu_n:=\mu \vert Y_n$, that is, for any  Borel set $ A\in \mathcal{B}(X)$, $\mu_n(A)=\frac{\mu(A\cap Y_n)}{\mu(Y_n)}.$
Take $\varphi=1$ in Proposition 3.19, note that $\underline h_{\mu}^{BK}(f,d,\epsilon)=\underline h_{1,\mu}(f,\epsilon)$ and $\overline{{mdim}}_{M,f,Y_n}(0,d,\frac{\epsilon}{2})=R(f,d,Y_n,\frac{\epsilon}{2}).$ Then 
$$(1-\gamma(2\epsilon))\underline h_{\mu_n}^{BK}(f,d,\epsilon)\leq \overline{{mdim}}_{M,f,Y_n}(0,d,\frac{\epsilon}{2}).$$
Hence 
\begin{align*}
	\underline h_{\mu_n}^{BK}(f,d,\epsilon)&= \int_{Y_n}\liminf_{m \to \infty}-\frac{1}{m} \log \mu_n(B_m(x,\epsilon))d\mu_n\\
		&=\frac{1}{\mu(Y_n)} \int_{Y_n}\liminf_{m \to \infty}-\frac{1}{m} \log  \frac{\mu (B_m(x,\epsilon)\cap Y_n)}{\mu(Y_n)}d\mu	\\
		&\geq\frac{1}{\mu(Y_n)} \int_{Y_n}\liminf_{m \to \infty}-\frac{1}{m} \log \mu (B_m(x,\epsilon) d\mu.	
\end{align*}
Letting $n \to \infty$, we have 
\begin{align*}
\overline{{mdim}}_{M,f,Y}(0,d,\frac{\epsilon}{2})&\geq\lim_{n \to \infty}\overline{{mdim}}_{M,f,Y_n}(0,d,\frac{\epsilon}{2})\\
&\geq \lim_{n \to \infty}(1-\gamma(2\epsilon))\underline h_{\mu_n}^{BK}(f,d,\epsilon)\\
&\geq (1-\gamma(2\epsilon))\underline h_{\mu}^{BK}(f,d,\epsilon).
\end{align*}
This implies that $\limsup_{\epsilon \to 0}\limits \frac{\underline h_{\mu}^{BK}(f,d,\epsilon)}{\log \frac{1}{\epsilon}}\leq \overline{{mdim}}_{M}(f,Y,d).$
	
\end{proof}
\begin{corollary}
	Let $(X,f)$ be a TDS with  a metric $d\in \mathscr{D}(X)$ and $\mu \in E(X,f)$.
	Then 
	
	$$\limsup_{\epsilon \to 0} \frac{\underline h_{\mu}^{BK}(f,d,\epsilon)}{\log \frac{1}{\epsilon}}\leq \overline{{mdim}}_{M}(f,G_\mu,d).$$
\end{corollary}

\begin{proposition}
	Let $(X,f)$ be a TDS  with  a metric $d\in \mathscr{D}(X)$  and $\mu \in E(X,f)$. Then for each $\epsilon>0$, 
	$$\overline h^K_{\mu}(f,d,2\epsilon)\leq \overline h_{\mu}^{BK}(f,d,\epsilon). $$
\end{proposition}

\begin{proof}
	Fix $\epsilon>0$ and let $F(x,\epsilon)=\limsup_{n \to \infty} \limits -\frac{\log\mu(B_n(x,\epsilon))}{n}$. Since 
	$f(B_{n+1}(x,\epsilon))\subset B_n(f(x),\epsilon)$ for all $x\in X$ and $n\in \mathbb{N}$ and $\mu \in M(X,f)$, we have
	$$\mu(B_{n+1}(x,\epsilon)) \leq \mu(f^{-1}f(B_{n+1}(x,\epsilon)))=\mu(f(B_{n+1}(x,\epsilon)))\leq \mu(B_n(f(x),\epsilon)).$$
	This yields that  $F(f(x),\epsilon)\leq F(x,\epsilon)$. It follows from  $\mu \in E(X,f)$ that   $F(x,\epsilon)$ is a constant  $\mu$-a.e $x \in X$. 
  Let $s>\overline h_{\mu}^{BK}(f,d,\epsilon)$.  Fix $\delta \in (0,1)$ and set
	$$X_N:=\{x\in X: \mu (B_n(x,\epsilon))>e^{-ns}, \forall n\geq N\}.$$ 
	Then we have $\mu(\cup_{N\geq 1}X_N)=1$. There exists  $N_0$ such that for any $N\geq N_0$, $\mu (X_N)>1-\delta$.  For each $N\geq N_0$, let $E_N$ be the $(N,2\epsilon)$-separated  set of $X_N$ with maximal cardinality. Note that the Bowen balls  $B_N (x,\epsilon), x\in E_N$ are  pairwise disjoint, hence  we have 
	
	$$1\geq \mu(\cup_{x\in E_N}(B_N(x,\epsilon)))=\sum_{x\in E_N}\mu(B_N(x,\epsilon))>\sum_{x\in E_N} e^{-Ns}. $$
	Then $r_N(\mu;d,2\epsilon,\delta)\leq\#E_N\leq e^{Ns}$ for all $N\geq N_0$, which implies  $\overline h^K_{\mu}(f,d,2\epsilon)\leq\overline h_{\mu}^{BK}(f,d,\epsilon)$.
\end{proof}

\begin{proposition}\rm {\cite[Proposition 4.3, Proposition 6.1]{w21}}
	Let $(X,f)$ be a TDS   with  a metric $d\in \mathscr{D}(X)$ and $\mu \in E(X,f)$. Then for each  $\epsilon>0$, 
	$$R_{ L^{\infty}}(d,\mu,\epsilon)\leq \overline  h^K_{\mu}(f,d,\epsilon)\leq PS(f,d,\mu,\epsilon)\leq R_{ L^{\infty}}(d,\mu,\frac{1}{6}\epsilon)$$
	 and  $\overline{{mdim}}_{M,f,G_\mu}(0,d,\epsilon) \leq PS(f,d,\mu,\epsilon)$.
\end{proposition}
Finally, we give the proof of Theorem 1.5,(ii).
\begin{proof}[Proof of Theorem \ref{thm 1.5}, (ii)]
	By virtue of Proposition 3.22 and Proposition 3.23,
	 we have 
	 \begin{align*}
	 	\overline{{mdim}}_{M}(f,G_\mu, d,12\epsilon)&= \overline{{mdim}}_{M,f,G_\mu}(0,d,12\epsilon)\\ 
	 	&\leq PS(f,d,\mu,12\epsilon)\\
	 	&\leq R_{ L^{\infty}}(d,\mu,2\epsilon)\\
	 	&\leq \overline h^K_{\mu}(f,d,2\epsilon)\\
	 	&\leq \overline h_{\mu}^{BK}(f,d,\epsilon).
	 \end{align*}
 Combining the assumption  $\limsup_{\epsilon \to 0}\limits \frac{\underline h_{\mu}^{BK}(f,d,\epsilon)}{\log \frac{1}{\epsilon}}=\limsup_{\epsilon \to 0}\limits \frac{\overline h_{\mu}^{BK}(f,d,\epsilon)}{\log \frac{1}{\epsilon}}$ and  Corollary 3.21, this completes the proof.

\end{proof}

One says that a compact metric space $(X,d)$ admits \emph{tame growth of covering numbers} if for each $\theta>0$,
$$\lim_{\epsilon \to 0}\epsilon^{\theta}\log r_1(f,d,\epsilon,X)=0.$$

This condition was introduced by  Lindenstrauss  and Tsukamoto \cite{lt18}  to show the metric mean dimensions defined by Bowen metric and average metric  coincide, see \cite[Lemma 26]{lt18},  which was proved  as a fairly  mild condition \cite[Lemma 4]{lt18}. Namely, every compact metrizable space  admits  a  metric with the property of tame growth of covering numbers,  and  some examples satisfying such condition  can be found in \cite[Example 3.9]{lt19}.

Under the assumption of   tame growth of covering numbers,  Wang \cite [Theorem 1.7]{w21}  also showed   if  $\mu \in E(X,f)$, then 
$$\overline{rdim}(X,f,d,\mu)=\overline{rdim}_{L^{\infty}}(X,f,d,\mu).$$

 Together with Theorem 1.5,  we immediately deduce the following.
 
\begin{corollary}Let $(X,f)$ be a TDS with a  metric $d\in \mathscr{D}(X)$  admitting tame growth of covering numbers,  and suppose that $\mu \in E(X,f)$ satisfying $\limsup_{\epsilon \to 0} \limits \frac{\overline h_{\mu}^{BK}(f,d,\epsilon)}{\log \frac{1}{\epsilon}}=\limsup_{\epsilon \to 0}\limits \frac{\underline h_{\mu}^{BK}(f,d,\epsilon)}{\log \frac{1}{\epsilon}}$. Then 
	$$\overline{{mdim}}_{M}(f,G_\mu,d)=\overline{rdim}(X,f,d,\mu)=\overline{rdim}_{L^{\infty}}(X,f,d,\mu).$$ 
\end{corollary}
\begin{example}
	Let $\sigma:[0,1]^{\mathbb{Z}}\rightarrow [0,1]^{\mathbb{Z}}$ be the shift on alphabet $[0,1]$, where $[0,1]$ is the unit interval with the standard metric. Equipped $ [0,1]^{\mathbb{Z}}$ with a metric given by 
	$$d(x,y)=\sum_{n\in \mathbb{Z}}2^{-|n|}|x_n-y_n|.$$
	Then $([0,1]^{\mathbb{Z}},d)$  has the tame growth of covering numbers, see \cite[Example 3.9]{lt19}. 
	Let $\mu=\mathcal{L}^{\otimes \mathbb{Z}}$, where $\mathcal{L}$ is the Lebesgue measure on $[0,1]$. 
	
	For each $\epsilon>0$, $x\in [0,1]^{\mathbb{Z}}$. Let $r=\left \lceil  \log_2\frac{4}{\epsilon}\right \rceil+1$. Then $\sum_{|n|>r}2^{-|n|}<\frac{\epsilon}{2}$.  Put
	\begin{align*}
		I_n(x,\epsilon):&=\left\{y\in [0,1]^{\mathbb{Z}}:|x_i-y_i|< \frac{\epsilon}{6},\forall -r\leq i\leq n+r \right\},\\
		J_n(x,\epsilon):&=\left\{y\in [0,1]^{\mathbb{Z}}:|x_i-y_i|< \epsilon, \forall0 \leq i\leq n\right\}.
	\end{align*}
	One can check that 
	$$	I_n(x,\epsilon)\subset B_n(x,\epsilon)\subset J_n(x,\epsilon).$$
	It is clear that  $\mu(I_n(x,\epsilon))\geq (\frac{\epsilon}{6})^{n+2r}, \mu(J_n(x,\epsilon))\leq (4\epsilon)^n$. This implies that 
	$$\log \frac{1}{4\epsilon}\leq\underline{h}_{\mu}^{BK} (\sigma,d,\epsilon)\leq \overline{h}_{\mu}^{BK} (\sigma,d,\epsilon)\leq \log \frac{6}{\epsilon},$$
	which tells us that $\limsup_{\epsilon \to 0}\limits  \frac{\overline h_{\mu}^{BK}(\sigma,d,\epsilon)}{\log \frac{1}{\epsilon}}=\limsup_{\epsilon \to 0} \limits \frac{\underline h_{\mu}^{BK}(\sigma,d,\epsilon)}{\log \frac{1}{\epsilon}}=1$.
	
	It is well-known that $\overline{{mdim}}_{M}(\sigma,[0,1]^{\mathbb{Z}},d)=1$, see \cite[Section II, E. Example]{lt18}. By Corollary 3.21, $1\leq\overline{{mdim}}_M(\sigma,G_{\mu},d)\leq\overline{{mdim}}_M(\sigma,[0,1]^{\mathbb{Z}},d)=1.$
	So  $\overline{{mdim}}_M(\sigma,G_{\mu},d)=1$.
	By \cite[Example 22]{lt18}, we know 
	$\overline{rdim}([0,1]^{\mathbb{Z}},\sigma,d,\mu)=\overline{rdim}_{L^{\infty}}([0,1]^{\mathbb{Z}},\sigma,d,\mu)=1.$ Finally, 
	$$\overline{{mdim}}_M(\sigma,G_{\mu},d)=\overline{rdim}([0,1]^{\mathbb{Z}},\sigma,d,\mu)=\overline{rdim}_{L^{\infty}}([0,1]^{\mathbb{Z}},\sigma,d,\mu)=1.$$
	
\end{example}

\subsection*{Acknowledgements}
\noindent We would like to thank anonymous referees and editors for useful comments on this article. The work was supported by the
National Natural Science Foundation of China (Nos.12071222 and 11971236), China Postdoctoral Science Foundation (No.2016M591873),
and China Postdoctoral Science Special Foundation (No.2017T100384). The work was also funded by the Priority Academic Program Development of Jiangsu Higher Education Institutions.  We would like to express our gratitude to Tianyuan Mathematical Center in Southwest China(No. 11826102), Sichuan University and Southwest Jiaotong University for their support and hospitality.


\begin{thebibliography}{HD82}

\normalsize
\baselineskip=15pt

%%%%%%%%%%%%%%%%%%%%%%%%%%%%%%%%
\bibitem[B96]{b96} L.  Barreira,  A non-additive thermodynamic formalism and applications to dimension theory of hyperbolic dynamical systems, \emph{Ergodic Theory Dynam. Syst.} \textbf{16} (1996), 871-927.

\bibitem[B73]{b73}R. Bowen,  Topological entropy for noncompact sets,  \emph {Trans. Amer. Math. Soc.} \textbf{184} (1973), 125-136.

\bibitem[B79]{b79}R. Bowen,  Hausdorff dimension of quasi-circles,  \emph {Inst. Hautes Études Sci. Publ. Math.} \textbf{50} (1979), 11-25.


\bibitem[B83]{b83} M. Boyle,  Lower entropy factors of sofic systems, \emph {Ergodic Theory Dynam. Syst.} \textbf{3} (1983), 541-557.


\bibitem[BK83]{bk83}M. Brin and A. Katok,  On local entropy,  Geometric dynamics (Rio de Janeiro),  \emph {Lecture Notes in Mathematics},  Springer, Berlin,  \textbf{1007} (1983), 30-38.

\bibitem[BS00]{bs00} L.  Barreira  and J. Schmeling, Sets of "non-typical" points have full topological entropy and full Hausdorff dimension, \emph{Israel J. Math.} \textbf{116} (2000), 29-70.

\bibitem[CT06]{ct06} T. M. Cover and J. A. Thomas, Elements of Information Theory, second edition, Wiley, New York, 2006.

\bibitem[C11]{c11} V. Climenhaga,  Bowen's equation in the non-uniform setting, {Ergodic Theory Dynam. Syst.} \textbf{31} (2011), 1163-1182.



\bibitem[CLS21]{cls21} D. Cheng, Z. Li and B. Selmi, Upper metric mean dimensions with potential on subsets, \emph {Nonlinearity} \textbf{34} (2021), 852-867.


\bibitem[CDZ22]{chen} E. Chen, D. Dou and D. Zheng, Variational principles for amenable metric mean dimensions, \emph {J. Diff. Equ.} \textbf{319} (2022), 41-79.
\bibitem[D06]{d06} T. Downarowicz, Minimal models for noninvertible and not uniquely ergodic systems, \emph{Israel J. Math.} \textbf{156} (2006), 93–110.

\bibitem[FH12]{fh12} D. J. Feng and W. Huang, Variational principles for topological entropies of subsets, \emph {J. Funct. Anal.} \textbf{263} (2012), 2228-2254.


\bibitem[Gro99]{gromov} M. Gromov,  Topological invariants of dynamical systems and spaces of holomorphic maps: I,  \emph{Math. Phys, Anal. Geom.} \textbf{4} (1999), 323-415.

\bibitem[G15]{g15} Y. Gutman, Mean dimension and Jaworski-type theorems, \emph{Proc. Lond. Math. Soc.} \textbf{111} (2015),  831-850.

\bibitem[GLT16]{glt16} Y. Gutman, E. Lindenstrauss and M. Tsukamoto, Mean dimension of $\mathbb{Z}^{k}$ actions, \emph{Geom. Funct. Anal.} \textbf{26} (2016),  778-817.

\bibitem[G17]{g17} Y. Gutman, Embedding topological dynamical systems with periodic points in cubical shifts, \emph{Ergodic Theory Dynam. Syst.} \textbf{37} (2017), 512-538.


\bibitem[GS20]{gs18}  Y. Gutman and A. $\rm \acute{\ S}$piewak, Metric mean dimension and analog compression, \emph {IEEE Trans. Inform. Theory} \textbf{66} (2020), 6977-6998.

\bibitem[GT20]{gt20}  Y. Gutman and M. Tsukamoto, Embedding minimal dynamical systems into Hilbert cubes, \emph {Invent. Math.} \textbf{221} (2020), 113-166.

\bibitem[GS21]{gs20} Y. Gutman and A. $\rm \acute{\ S}$piewak,  Around the variational principle for metric mean dimension,  \emph {Studia Math.}  \textbf{261} (2021), 345-360.

\bibitem[JKL14]{jms14} J. Jaerisch,  M. Kesseböhmer and S. Lamei, Induced topological pressure for countable state Markov shifts, \emph{Stoch. Dyn.} \textbf{14} (2014), 1350016, 31 pp.



\bibitem[K80]{k80} A. Katok,  Lyapunov exponents, entropy and periodic orbits for diffeomorphisms, \emph{Publ. Math. Inst. Hautes Études Sci.} \textbf{51} (1980), 137-173.

\bibitem[K82]{k82} W. Krieger,  On the subsystems of topological Markov chains, \emph {Ergodic Theory Dynam. Syst.} \textbf{2} (1982), 195-202.



\bibitem[LT18]{lt18} E. Lindenstrauss  and M. Tsukamoto, From rate distortion theory to metric mean dimension: variational principle, \emph{IEEE Trans. Inform. Theory}  \textbf{64} (2018), 3590-3609.

\bibitem[LT19]{lt19} E.  Lindenstrauss and M. Tsukamoto,  Double variational principle for mean dimension, \emph{Geom. Funct. Anal.} \textbf{29} (2019),  1048-1109.

\bibitem[L99]{l99} E. Lindenstrauss,  Mean dimension, small entropy factors and an embedding theorem, \emph{Publ. Math. Inst. Hautes Études Sci.} \textbf{89} (1999), 227-262.

\bibitem[LW00]{lw00} E. Lindenstrauss and B. Weiss,  Mean topological dimension,  \emph{Israel J. Math.} \textbf{115} (2000), 1-24.  

\bibitem[M95]{m95} P. Mattila,  The Geometry of Sets and Measures in Euclidean Spaces, \emph{Cambridge University Press}, 1995.  

\bibitem[PS07]{ps07} C. E. Pfister and W. G. Sullivan, On the topological entropy of saturated sets, \emph {Ergodic Theory Dynam. Syst.}  \textbf{27} (2007), 929-956.

\bibitem[S21]{shi} R. Shi, On variational principles for metric mean dimension, to appear in \emph{ IEEE Trans. Inform. Theory}, https://doi.org/10.1109/TIT.2022.3157786.

\bibitem[T20]{t20} M. Tsukamoto, Double variational principle for mean dimension with potential, \emph{Adv. Math.} \textbf{361} (2020), 106935, 53 pp.

\bibitem[VV17]{vv17}  A. Velozo and R. Velozo,   Rate distortion theory, metric mean dimension and measure theoretic entropy, arXiv:1707.05762.

\bibitem[W75]{w75}  P. Walters, A variational principle for the pressure of continuous transformations, \emph{Amer. J. Math.} \textbf{97} (1975),  937-971.

\bibitem[WC12]{wc12} C. Wang and E. Chen, Variational principles for BS dimension of subsets, \emph{Dyn. Syst.}  \textbf{27} (2012),  359-385.



\bibitem[XC15]{xc} Z.  Xing  and E. Chen, Induced topological pressure for topological dynamical systems, \emph{J. Math. Phys.} \textbf{56} (2015), 022707, 10 pp.

\bibitem[ZC18]{zc18} D. Zheng  and E. Chen, Topological entropy of sets of generic points for actions of amenable groups, \emph{Sci. China Math.} \textbf{61} (2018), 869-880.

\bibitem[W21]{w21} T. Wang, Variational relations for metric mean dimension and rate distortion dimension, \emph{Discrete Contin. Dyn. Syst.} \textbf{27} (2021),  4593-4608.

\bibitem[P97]{p97} Y.B. Pesin, Dimension theory in dynamical systems, University of Chicago Press, 1997.
\bibitem[W82]{walter} P. Walter, An introduction to ergodic theory,  Springer-Verlag, New York, 1982.

\end{thebibliography}
\end{document}